\documentclass[english,11pt]{article}

\usepackage{amssymb,amsmath,amsfonts,amsthm,rotating,setspace,graphicx,float}

\usepackage{graphicx}

\newtheorem{theorem}{Theorem}[section]
\newtheorem{corollary}[theorem]{Corollary}
\newtheorem{prop}[theorem]{Proposition}
\newtheorem{lemma}[theorem]{Lemma}

\theoremstyle{definition}
\newtheorem{rmk}{Remark}

\numberwithin{equation}{section}

\def\neweq#1{\begin{equation}\label{#1}}
\def\endeq{\end{equation}}

\def\ep{\varepsilon}
\def\vphi{\varphi}
\def\la{\lambda}

\def\a{\alpha}
\def\b{\beta}
\def\intl{\int\limits}

\def\eq#1{(\ref{#1})}

\def\R{{\mathbb R} }

\begin{document}


\title{Asymptotic behavior at isolated singularities for solutions of nonlocal semilinear elliptic systems of inequalities}

\author{Marius Ghergu\footnote{School of Mathematical Sciences,
    University College Dublin, Belfield, Dublin 4, Ireland; {\tt
      marius.ghergu@ucd.ie}} $\;\;$ and $\;$ 
Steven D.~Taliaferro\footnote{Mathematics Department, Texas A\&M
    University, College Station, TX 77843-3368; {\tt stalia@math.tamu.edu}} 
\footnote{Corresponding author, Phone 001-979-845-3261, Fax
  001-979-845-6028}
  }





\date{}


\maketitle

\begin{abstract}
We study the behavior near the origin of $C^2$ positive solutions $u(x)$ and $v(x)$ of the system
$$
\left\{
\begin{aligned}
0\leq -\Delta u\leq \left(\frac{1}{|x|^\alpha}* v\right) ^\lambda \\
0\leq -\Delta v\leq\left(\frac{1}{|x|^\beta}* u\right) ^\sigma
\end{aligned}
\quad \mbox{ in }B_2(0)\setminus\{0\}\subset \R^n, n\geq 3,
\right.
$$
where $\lambda,\sigma\geq 0$ and
$\alpha, \beta\in (0,n)$.

A by-product of our methods used to study these solutions will be
results on the behavior near the origin of $L^1(B_1(0))$ solutions $f$
and $g$ of the system
\[\left\{
\begin{aligned}
0\leq f(x)\leq C\left(|x|^{2-\alpha}
+\int_{|y|<1}\frac{ g(y)\,dy}{|x-y|^{\alpha-2}} \right)^\lambda\\
0\leq g(x)\leq C\left(|x|^{2-\beta}
+\int_{|y|<1}\frac{ f(y)\,dy}{|x-y|^{\beta-2}} \right)^\sigma
\end{aligned}
\right.\qquad\mbox{for }0<|x|<1
\]
where $\lambda,\sigma\geq 0$ and
$\alpha, \beta\in (2,n+2)$.
\medskip

\end{abstract}

\tableofcontents

\section{Introduction}

In this paper we study the behavior near the origin of
$C^2(\R^n\setminus\{0\})\cap L^1(\R^n)$ positive solutions $u(x)$
and $v(x)$  of the system
\neweq{nonloc}
\left\{
\begin{aligned}
0\leq -\Delta u\leq \left(\frac{1}{|x|^\alpha}* v\right) ^\lambda \\
0\leq -\Delta v\leq\left(\frac{1}{|x|^\beta}* u\right) ^\sigma
\end{aligned}
\quad \mbox{ in }B_2(0)\setminus\{0\}\subset \R^n, n\geq 3,
\right.
\endeq
where $\lambda, \sigma\geq 0$ and $\alpha, \beta\in (0,n)$.

The goal of this work is to address the following question.
\medskip

\noindent {\bf Question 1}. For which constants $\lambda, \sigma\geq
0$ and $\alpha, \beta\in (0,n)$ do there exist continuous functions
$h_1 ,h_2 :(0,1)\to(0,\infty)$ such that all
$C^2(\R^n\setminus\{0\})\cap L^1(\R^n)$ positive solutions $u(x)$
and $v(x)$ of the system \eqref{nonloc} satisfy
$$
 u(x)={\cal O}(h_1 (|x|)) \quad \text{ as } x\to 0
$$
$$
 v(x)={\cal O}(h_2 (|x|)) \quad \text{ as }x\to 0
$$
and what are the optimal such $h_1$ and $h_2$ when they exist?
\medskip

We call a function $h_1$ (resp. $h_2$) with the above properties a
{\it pointwise bound} for $u$ (resp. $v$) as $x\to0$.

\begin{rmk}\label{rem1} Let $\Gamma\in C^2(\R^n\setminus\{0\})\cap
  L^1(\R^n)$ be a positive function such that $\Gamma(x)=|x|^{-(n-2)}$
  for $0<|x|<2$. Since $-\Delta \Gamma =0$ in $B_2(0)\setminus\{0\}$,
  the functions $u_0 (x)=v_0 (x)=\Gamma(x)$ are always positive
  solutions of \eqref{nonloc}.  Hence, any pointwise bound for positive
  solutions of \eqref{nonloc} must be at least as large as $|x|^{-(n-2)}$ and
  whenever $|x|^{-(n-2)}$ is such a bound for $u$ (resp. $v$) it is
  necessarily optimal. In this case we say that $u$ (resp. $v$) is
  {\it harmonically bounded} at $0$.
\end{rmk}

A first motivation for the study of \eq{nonloc} comes from the
equation
\neweq{prototype}
-\Delta u=\Big(\frac{1}{|x|^\alpha}* u^{p}\Big) |u|^{p-2}u\quad\mbox{ in }\R^n,
\endeq
where $\alpha\in (0,n)$ and $p>1$.
For $n = 3$ and $\alpha=p=2$, equation \eq{prototype} is known in the literature as the {\it Choquard-Pekar equation}  and was introduced in \cite{P1954} as a model in quantum theory of a Polaron at rest (see also \cite{DA2010}). Later, the equation \eq{prototype} appears as a model of an electron trapped in its own hole, in an approximation to Hartree-Fock theory of one-component plasma \cite{L1976}.
More recently,  the same equation \eq{prototype} was used in a model of self-gravitating matter
(see, e.g., \cite{J1995,MPT1998}) and it is known in this context as the {\it Schr\"odinger-Newton equation}.
In the degenerate case $p=1$,  equation \eq{prototype} becomes the prototype for our system \eq{nonloc}.

Another motivation for the study of \eq{nonloc} is given by various integral equations that have been recently investigated. For instance, the system
\neweq{syst}
\left\{
\begin{aligned}
u(x)  = \left(\;\int_{\R^n} \frac{v(y)}{|x-y|^{\alpha}} dy\right)^\lambda  \\
 v(x)=\left(\;\int_{\R^n} \frac{u (y)}{|x-y|^{\beta}} dy\right)^\sigma
\end{aligned}
\qquad \mbox{ in }\R^n, \,n\geq 3,
\right.
\endeq
and its more general forms appear in \cite{CL2005, CL2009, JL2006, Leia2013, Leib2013, LLM2012}.
These works are mainly concerned with radial symmetry, monotonicity or regularity of solutions.

As emphasized in \cite{JL2006, LLM2012}, the system \eq{syst} is related to the Hardy-Littlewood-Sobolev inequality

\neweq{hls}
\left|\;\int_{\R^n}\int_{\R^n}\frac{f(x)g(y)}{|x-y|^\theta}dxdy\right|\leq C(n)\|f\|_p\|g\|_q,
\endeq
where $p,q>1$ and $\theta=(2-1/p-1/q)n$.

In order to find the best constant in \eq{hls} one has to find
$$
J:=\min_{\|f\|_p=1\,,\;\|g\|_q=1} \int_{\R^n}\int_{\R^n}\frac{f(x)g(y)}{|x-y|^\theta}dxdy
$$
and this leads to the Euler-Lagrange equations
\neweq{systt}
\left\{
\begin{aligned}
f(x)  = \left(\frac{1}{J}\int_{\R^n} \frac{g(y)}{|x-y|^{\theta}} dy\right)^\lambda  \\
 g(x)=\left(\frac{1}{J}\int_{\R^n} \frac{f (y)}{|x-y|^{\theta}} dy\right)^\sigma
\end{aligned}
\qquad \mbox{ in }\R^n,
\right.
\endeq
where $\lambda=1/(p-1)>0$ and $\sigma=1/(q-1)>0$.

We point out that \eq{nonloc} has a similar structure to
\eq{systt}. Indeed, by the well known result of Brezis and Lions
\cite{BL1981} (see Lemma \ref{brezislions} below) regarding the
representation of nonnegative superharmonic functions in the punctured
ball,  positive solutions $u$ and $v$ of \eq{nonloc} satisfy
$$
-\Delta u,-\Delta v\in L^1(B_1(0))
$$
and
\begin{equation*}
\left\{
\begin{aligned}
&u(x)\leq C\left(|x|^{2-n}+\int_{|y|<1}\frac{-\Delta
    u(y)\,dy}{|x-y|^{n-2}} \right) \\
&v(x)\leq C\left(|x|^{2-n}+\int_{|y|<1}\frac{-\Delta
    v(y)\,dy}{|x-y|^{n-2}} \right)
\end{aligned}
\quad \mbox{ for } 0<|x|<1.
\right.
\end{equation*}
Substituting these estimates in \eqref{nonloc} and using Lemma
\ref{uv} and Corollary \ref{corstein} below, we find for $\alpha,\beta\in
(2,n)$ that \neweq{hls_syst} \left\{
\begin{aligned}
0\leq f(x)\leq M\left(|x|^{2-\alpha}
+\int_{|y|<1}\frac{ g(y)\,dy}{|x-y|^{\alpha-2}} \right)^\lambda\\
0\leq g(x)\leq M\left(|x|^{2-\beta}
+\int_{|y|<1}\frac{ f(y)\,dy}{|x-y|^{\beta-2}} \right)^\sigma
\end{aligned}
\right.\qquad\mbox{for }0<|x|<1,
\endeq
where $f=-\Delta u$, $g=-\Delta v$ are $C(\R^n\setminus\{0\})\cap
L^1(B_1(0))$ functions and $M$ is a positive constant.

A by-product of our methods used to study solutions of
\eqref{nonloc} will be results on the behavior near the origin of
$L^1(B_1(0))$ solutions $f$ and $g$ of \eqref{hls_syst} when
$\la,\sigma\ge 0$ and $\a,\b\in(2,n+2)$.

Before we state the main results for \eq{nonloc} let us mention the
following system which we considered in \cite{GTV2014}:
\neweq{recent}
\left\{
\begin{aligned}
0\leq -\Delta u\leq v^\lambda \\
0\leq -\Delta v\leq u^\sigma
\end{aligned}
\quad \mbox{ in }B_1(0)\setminus\{0\}\subset \R^n, n\geq 3, \right.
\endeq
where $\lambda, \sigma\geq 0$. In \cite{GTV2014} we emphasized the
existence of a critical curve in the $\lambda\sigma$-plane that
optimally describes the existence of pointwise bounds for
\eq{recent}. A particular feature of \eq{recent} is that whenever
pointwise bounds exist, then at least one of $u$ and $v$ must be
harmonically bounded. We shall see that this is not always the case
when dealing with the nonlocal system \eq{nonloc}. Theorems
\ref{thm4B} and \ref{optimalThm2} below illustrate such a phenomenon
which we believe is due to the more complex character of \eq{nonloc}
that involves four parameters $\alpha, \beta, \lambda, \sigma$
(instead of two parameters in the case of \eq{recent}).

Since positive solutions $u$ and $v$ of the system of inequalities 
\eqref{nonloc} (resp. \eqref{recent}) are also solutions of the
system of equations
\begin{equation}\label{nonloc-eq}
\left\{
\begin{aligned}
-\Delta u=\left(\frac{1}{|x|^\alpha}* v\right) ^\lambda \\
-\Delta v=\left(\frac{1}{|x|^\beta}* u\right) ^\sigma
\end{aligned}\right.
\quad \mbox{ in }B_2(0)\setminus\{0\}\subset \R^n, n\geq 3,
\end{equation}
\begin{equation}\label{recent-eq}
\left(\text{resp. }\left\{\begin{aligned}
-\Delta u=v^\lambda \\
-\Delta v=u^\sigma
\end{aligned}\right.
\quad \mbox{ in }B_1(0)\setminus\{0\}\subset \R^n, n\geq 3, \right)
\end{equation}
our pointwise bounds at 0 for solutions of the systems \eqref{nonloc} and
\eqref{recent} also hold for
solutions of the systems \eqref{nonloc-eq} and \eqref{recent-eq}
respectively. Such bounds are often a first step for obtaining more
precise asymptotic behavior at 0  of positive solutions of
systems \eqref{nonloc-eq} and \eqref{recent-eq}
and nonexistence of entire solutions.
The system \eqref{recent-eq} has been studied extensively. See for
example \cite{BR1996} and \cite{S2009}.

\section{Statement of the main results}
\subsection{Results for system \eqref{nonloc}}\label{pde}
We first consider the case that either $\alpha$ or $\beta$ belongs to
the interval $(0,2]$. We can assume without loss of generality
that $\beta\in (0,2]$.

\begin{theorem}\label{thm1}
Suppose
\[
\alpha\in (0,n),\quad \beta\in (0,2],\quad \text{and} \quad
\lambda,\sigma\ge 0.
\]
Let $u$ and $v$ be $C^2(\R^n\setminus\{0\})\cap
L^1(\R^n)$ positive solutions of \eqref{nonloc}. Then
\neweq{thm1estu}
u(x)=
\left\{
\begin{aligned}
& {\mathcal O}(|x|^{-(n-2)}) &\quad\mbox{ if }n\geq \lambda(\alpha-2),\\
&o\Big(|x|^{-\frac{\lambda(\alpha-2)(n-2)}{n}}\Big) &\quad\mbox{ if } n< \lambda(\alpha-2),
\end{aligned}
\right. \quad\mbox{ as }x\to 0,
\endeq
and
\neweq{thm1estv}
v(x)={\mathcal O}(|x|^{-(n-2)}) \quad\mbox{ as }x\to 0.
\endeq
\end{theorem}

By Remark \ref{rem1} the estimate \eqref{thm1estv} and
the first estimate in \eqref{thm1estu} are optimal. By the following
theorem, the second estimate in \eqref{thm1estu} is also optimal.

\begin{theorem}\label{optimalthm1}
Suppose
$$
0<\beta\leq 2<\alpha<n\quad\mbox{ and }\quad\lambda>\frac{n}{\alpha-2}.
$$
Let $h:(0,1)\to (0,1)$ be a continuous function satisfying
$\lim_{t\to 0^+}h(t)=0$. Then there exist
$C^\infty(\R^n\setminus\{0\})\cap L^1(\R^n)$ positive solutions  $u$ and $v$ of
\eq{nonloc} such that
\neweq{optimalesti1} u(x)\neq {\mathcal
  O}\Big(h(|x|) |x|^{-\frac{\lambda(\alpha-2)(n-2)}{n}}\Big)
\quad\mbox{ as }x\to 0
\endeq
and
\neweq{optimalesti2}
v(x)|x|^{n-2}\to 1 \quad\mbox{ as }x\to 0.
\endeq
\end{theorem}

Note that, according to Theorem \ref{thm1}, if $\alpha,\beta\in (0,2]$
then all positive solutions $u$ and $v$ of \eqref{nonloc} are
harmonically bounded, that is
\begin{equation}\label{asyhar}
u(x)={\cal O}(|x|^{-(n-2)})\quad \text{and}\quad
v(x)={\cal O}(|x|^{-(n-2)})\quad\text{as }x\to 0,
\end{equation}
regardless of the size of the exponents $\lambda$ and $\sigma$.

We next consider the case that
$\alpha,\beta\in (2,n)$. In this setting the study of the asymptotic
behavior is more delicate and it involves all parameters $\alpha,
\beta,\lambda$ and $\sigma$.
We can assume without loss of generality that
\begin{equation}\label{quadrant}
0\le(\beta-2)\sigma\le (\alpha-2)\lambda.
\end{equation}

Let $\alpha,\beta\in (2,n)$ be fixed constants. If $\lambda$ and
$\sigma$ satisfy \eqref{quadrant} then $(\lambda,\sigma)$ belongs to
one of the following five pairwise disjoint subsets of the
$\lambda\sigma$-plane.
\begin{align*}
 &A:=\left\{ (\lambda ,\sigma): \, 0\le \lambda\le \frac{n}{\alpha-2}
   \quad\text{and}\quad
0\leq \sigma
\leq \frac{\alpha-2}{\beta-2}\lambda \right\}
\setminus \left\{\left(\frac{n}{\alpha-2},\frac{n}{\beta-2}\right)\right\}\\
 &B:=\left\{ (\lambda ,\sigma): \, \lambda>\frac{n}{\alpha-2}
   \quad\text{and}
\quad
0\le\sigma\le \frac{2}{\beta-2}
+\frac{n(n-2)}{(\alpha-2)(\beta-2)} \frac{1}{\lambda} \right\}\\
&C:=\left\{ (\lambda ,\sigma): \, \lambda>\frac{n}{\alpha-2}
   \quad\text{and}
\quad
\frac{2}{\beta-2}+\frac{n(n-2)}{(\alpha-2)(\beta-2)} \frac{1}{\lambda}
<\sigma<\frac{n+2-\alpha}{\beta-2}
+\frac{n}{\beta-2} \frac{1}{\lambda} \right\} \\
&D:=\left\{ (\lambda ,\sigma): \, \lambda>\frac{n}{\alpha-2} \quad
   \text{and}\quad \frac{n+2-\alpha}{\beta-2}+\frac{n}{\beta-2} \frac{1}{\lambda}<\sigma \leq \frac{\alpha-2}{\beta-2}\lambda \right\} \\
 &E:=\left\{ (\lambda ,\sigma): \, \lambda\ge\frac{n}{\alpha-2} \quad
   \text{and}\quad \sigma=\frac{n+2-\alpha}{\beta-2}+\frac{n}{\beta-2} \frac{1}{\lambda} \right\}.
\end{align*}

\begin{figure}[H]
 \includegraphics[scale=.55]{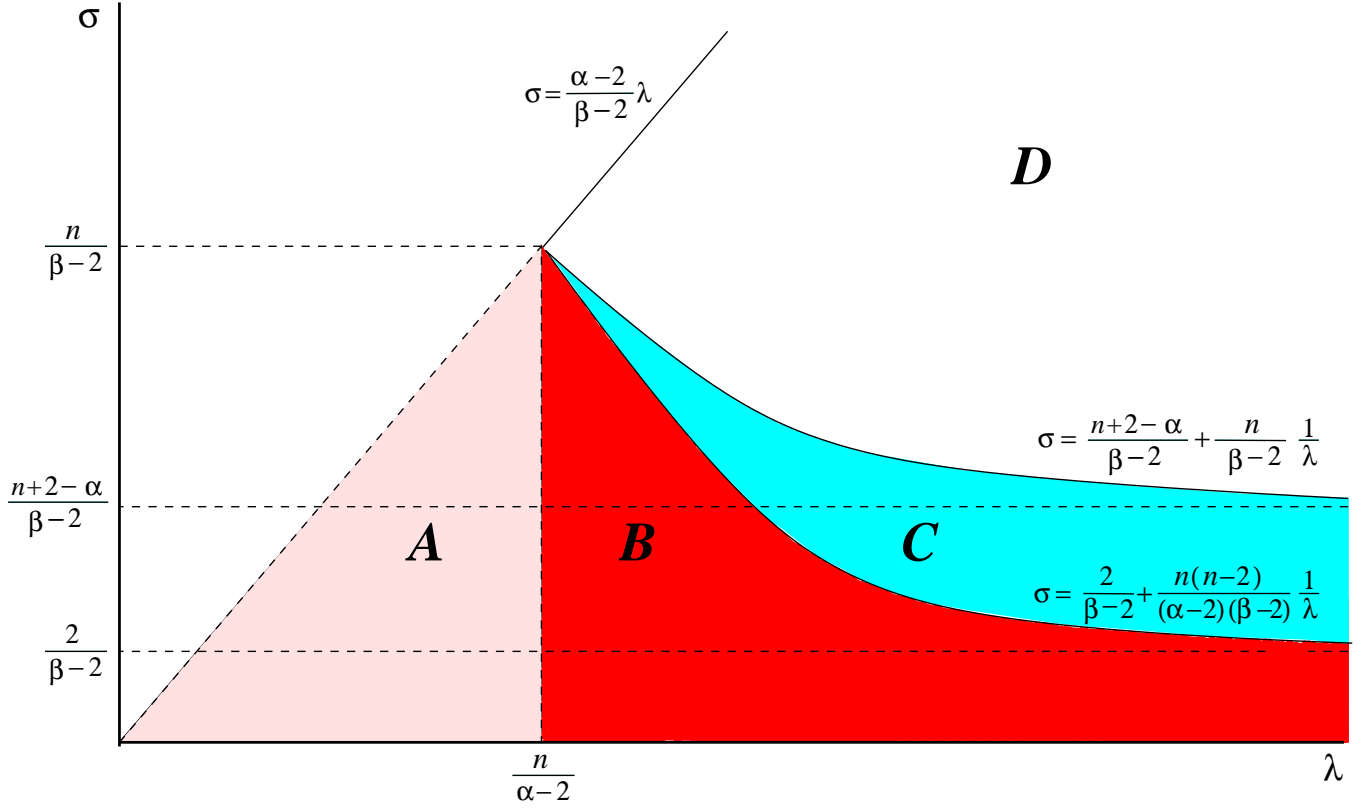}
 \caption{Graph of regions $A$, $B$, $C$ and $D$.}
\end{figure}

Note that $A$, $B$, $C$, and $D$ are two dimensional regions in the $\lambda
\sigma$-plane whereas $E$ is the curve separating $C$ and $D$. (See
Figure 1.)

The following theorem deals with the case that $(\lambda,\sigma)\in A$.

\begin{theorem}\label{thm3}
Suppose $\alpha,\beta\in (2,n)$,
\begin{equation}\label{assumption}
0 \le \lambda\leq\frac{n}{\alpha-2}
\quad \text{and} \quad 0\le \sigma< \frac{n}{\beta-2}.
\end{equation}
Let $u$ and $v$ be $C^2(\R^n\setminus\{0\})\cap L^1(\R^n)$ positive
solutions of \eq{nonloc}. Then $u$ and $v$ are both harmonically bounded,
that is $u$ and $v$ satisfy \eqref{asyhar}.
\end{theorem}

By Remark \ref{rem1} the bounds \eqref{asyhar} for $u$ and $v$ in
Theorem \ref{thm3} are optimal.

The following theorem deals with the case that
$(\lambda,\sigma)\in B$.

\begin{theorem}\label{thm4A}
Suppose $\alpha,\beta\in (2,n)$,
\[
\lambda>\frac{n}{\alpha-2}\quad\mbox{ and }\quad
0\leq \sigma \le \frac{2}{\beta-2}+\frac{n(n-2)}{(\alpha-2)(\beta-2)} \frac{1}{\lambda}.
\]
Let $u$ and $v$ be $C^2(\R^n\setminus\{0\})\cap L^1(\R^n)$ positive
solutions of \eq{nonloc}. Then
\[
u(x)=o\Big(|x|^{-\frac{\lambda(\alpha-2)(n-2)}{n}}\Big) \quad\mbox{ as }x\to 0
\]
and $v$ is harmonically bounded, that is
\[
v(x)={\mathcal O}(|x|^{-(n-2)})\quad\mbox{ as }x\to 0.
\]
\end{theorem}

The following theorem deals with the case that
$(\lambda,\sigma)\in C$.

\begin{theorem}\label{thm4B}
Suppose $\alpha,\beta\in (2,n)$,
\[
\lambda>\frac{n}{\alpha-2}\quad\mbox{ and }\quad
\frac{2}{\beta-2}+\frac{n(n-2)}{(\alpha-2)(\beta-2)} \frac{1}{\lambda}
<\sigma<\frac{n+2-\alpha}{\beta-2}
+ \frac{n}{\beta-2}\frac{1}{\lambda}.
\]
Let $u$ and $v$ be $C^2(\R^n\setminus\{0\})\cap L^1(\R^n)$ positive
solutions of \eq{nonloc}. Then
\[
u(x)=o\Big(|x|^{-\frac{\lambda(\alpha-2)(n-2)}{n}}\Big) \quad\mbox{ as }x\to 0
\]
and
\[
v(x)=o\Big(|x|^{-\frac{\lambda(\alpha-2)[\sigma(\beta-2)-2]}{n}}\Big)
\quad\mbox{ as }x\to 0.
\]
\end{theorem}

By the following theorem the bounds for $u$ and $v$ in Theorems
\ref{thm4A} and \ref{thm4B} are optimal.

\begin{theorem}\label{optimalThm2}
Suppose $\alpha,\beta\in(2,n)$,
\[
\lambda>\frac{n}{\alpha-2}, \quad\text{and}\quad 0<\sigma<\frac{n}{\beta-2}.
\]
Let $h:(0,1)\to(0,1)$ be a continuous function such that $\lim_{t\to
  0^+}h(t)=0$. Then there exist
$C^\infty(\R^n\setminus\{0\})\cap L^1(\R^n)$ positive solutions $u$ and
$v$ of \eqref{nonloc} such that
\[
u(x)\not= {\cal O}\left(h(|x|)|x|^{-\frac{\lambda(\alpha-2)(n-2)}{n}}\right)
\quad\text{as }x\to 0
\]
and
\[
v(x)\not={\cal O}\left(h(|x|)\left[|x|^{-(n-2)}+|x|^{-\frac{\lambda(\alpha-2)[\sigma(\beta-2)-2]}{n}}\right]\right)\quad\text{as }x\to 0.
\]
\end{theorem}

The following theorem deals with the case that $(\lambda,\sigma)\in
D$. In this case there exist pointwise bounds for neither $u$ nor
$v$.

\begin{theorem}\label{thm5}
Suppose $\alpha,\beta\in (2,n)$,
\neweq{optimm1}
\lambda>\frac{n}{\alpha-2}\quad\mbox{ and }\quad
\sigma>\frac{n+2-\alpha}{\beta-2}
+ \frac{n}{\beta-2}\frac{1}{\lambda}.
\endeq
Let $h:(0,1)\to (0,\infty)$ be a continuous function such that
$\lim_{t\to0^+}h(t)=\infty$. Then there exist
$C^\infty(\R^n\setminus\{0\})\cap L^1(\R^n)$ positive solutions $u$ and $v$
of \eq{nonloc} such that
\neweq{large1} u(x)\neq {\mathcal
  O}(h(|x|))\quad\mbox{ as }x\to 0
\endeq
and
\neweq{large2}
v(x)\neq {\mathcal O}(h(|x|))\quad\mbox{ as }x\to 0.
\endeq
\end{theorem}

From Theorems \ref{thm1}, \ref{thm3} and \ref{thm5} we find:

\begin{corollary}
Let $\alpha\in (0,n)$ and $\lambda\geq 0$. Consider the inequality
\neweq{corsimple}
0\leq -\Delta u\leq \left(\frac{1}{|x|^\alpha}* u\right) ^\lambda
\quad \mbox{ in }B_2(0)\setminus\{0\}\subset \R^n, n\geq 3.
\endeq
\begin{enumerate}
\item[(i)] If $\lambda(\alpha-2)< n$ then any
  $C^2(\R^n\setminus\{0\})\cap L^1(\R^n)$ positive solution $u$ of
  \eq{corsimple} satisfies
$$
u(x)={\mathcal O}(|x|^{-(n-2)}) \quad\mbox{ as }x\to 0.
$$
\item[(ii)] If $\lambda(\alpha-2)> n$ then \eq{corsimple} admits
  $C^2(\R^n\setminus\{0\})\cap L^1(\R^n)$ positive solutions which are
  arbitrarily large around the origin in the following sense: for any
  continuous function $h:(0,1)\to (0,\infty)$ satisfying
  $\lim_{t\to0^+}h(t)=\infty$, there exists a
  $C^2(\R^n\setminus\{0\})\cap L^1(\R^n)$  positive solution $u$ of
  \eq{corsimple} such that
$$
u(x)\neq {\mathcal O}(h(|x|))\quad\mbox{ as }x\to 0.
$$
\end{enumerate}
\end{corollary}

A first tool we use in our approach to \eq{nonloc} is an integral
representation formula for nonnegative superharmonic functions in
punctured balls due to Brezis and Lions \cite{BL1981} (see also
\cite{CDM2008, GMT2011, T2007, T2011} where representation formulae
for various kinds of differential operators are deduced). Another
important tool in our approach is Proposition \ref{cor_nonlin} which
provides pointwise estimates for nonlinear potentials of
Havin-Maz'ya type. Further, various integral estimates will be
employed as stated in Section \ref{intg}. The optimality of the
pointwise bounds obtained in our main results will be achieved by
constructing solutions $u$ and $v$ of \eq{nonloc} satisfying suitable
coupled conditions on a countable sequence of balls that concentrate
at the origin. At this stage we leave open the question of
(non)existence of pointwise bounds for $(\lambda, \sigma)$ on the
curve $E$ defined above.

The remainder of the paper is organized as follows: In Subsection
\ref{integral} we state our main results for the system
\eqref{hls_syst}.  In Section \ref{prem} we collect various pointwise
and integral estimates for some quantities which will frequently
appear in the course of our proofs. Sections
\ref{sec-thm1}--\ref{sec-large-fg} contain the proofs of our main
results. Theorem \ref{crucial}, which deals with the system
\eqref{hls_syst}, is a crucial result, from which the optimal bounds
for positive solutions of the systems \eqref{nonloc} and
\eqref{hls_syst} easily follow.

\subsection{Results for system \eqref{hls_syst}}\label{integral}
We now state our results for the system \eqref{hls_syst} when
$\la,\sigma\ge 0$ and $\a,\b\in(2,n+2)$. As in Subsection~\ref{pde}, we
can assume that \eqref{quadrant} holds. Let the regions $A$--$D$ be
defined as in Subsection \ref{pde}.

The following theorem deals with the case $(\la,\sigma)\in A$.

\begin{theorem}\label{crucial1}
Suppose $\alpha,\beta\in(2,n+2)$,
\[
0\le \lambda\le \frac{n}{\a-2}\quad\text{ and }\quad
0\le \sigma< \frac{n}{\b-2}.
\]
Let $f$ and $g$ be $L^1(B_1(0))$ solutions
of \eqref{hls_syst} where $M$ is a positive constant. Then
\[
f(x)={\mathcal O}\left(|x|^{-\lambda(\alpha-2)}\right)\quad\mbox{ as }x\to 0
\]
and
\[
g(x)={\mathcal O}\left(|x|^{-\sigma(\beta-2)}\right)
\quad\mbox{ as }x\to 0.
\]
\end{theorem}

The following theorem deals with the case $(\la,\sigma)\in B\cup C$.

\begin{theorem}\label{crucial2}
Suppose $\alpha,\beta\in(2,n+2)$,
\[
\lambda> \frac{n}{\a-2}\quad\text{ and }\quad
0\le \sigma<\frac{n+2-\alpha}{\beta-2}
+\frac{n}{\beta-2} \frac{1}{\lambda}.
\]
Let $f$ and $g$ be $L^1(B_1(0))$ solutions
of \eqref{hls_syst} where $M$ is a positive constant. Then
\[
f(x)={\mathcal O}\left(|x|^{-\lambda(\alpha-2)}\right)\quad\mbox{ as
}x\to 0
\]
and
\[
g(x)=o\left(|x|^{-\frac{\lambda(\alpha-2)\sigma(\beta-2)}{n}}\right)
\quad\mbox{ as }x\to 0.
\]
\end{theorem}

By the following result the estimates
for $f$ and $g$ in Theorems \ref{crucial1} and \ref{crucial2} are optimal.

\begin{theorem}\label{optimal-fg}
Suppose $\varepsilon>0$, $\alpha,\beta\in(2,n+2)$,
\[
\lambda\ge 0 \quad\text{and}\quad 0<\sigma<\frac{n}{\beta-2}.
\]
Let $h:(0,1)\to (0,1)$ be a continuous function such that
$\lim_{t\to 0^+} h(t)=0$. Then there exist solutions
\begin{equation}\label{op1}
f,g\in C^\infty(\R^n\setminus\{0\})\cap L^1(\R^n)
\end{equation}
of the system
\neweq{op2} \left\{
\begin{aligned}
0\leq f(x)\leq \varepsilon\left(|x|^{-(\alpha-2)}
+\intl_{|y|<\varepsilon}\frac{ g(y)\,dy}{|x-y|^{\alpha-2}} \right)^\lambda\\
0\leq g(x)\leq \varepsilon\left(|x|^{-(\beta-2)}
+\intl_{|y|<\varepsilon}\frac{ f(y)\,dy}{|x-y|^{\beta-2}}
\right)^\sigma
\end{aligned}
\right.\qquad\mbox{for }x\in \R^n\setminus\{0\},\,n\ge 3,
\endeq
such that
\begin{equation}\label{op3}
\limsup_{x\to 0}|x|^{\lambda(\alpha-2)}f(x)>0
\end{equation}
and
\begin{equation}\label{op4}
g(x)\not={\cal
  O}\left(h(x)\left[|x|^{-\sigma(\beta-2)}+|x|^{-\frac{\lambda(\alpha-2)\sigma(\beta-2)}{n}}\right]\right)
\quad\text{as }x\to 0.
\end{equation}
\end{theorem}

The following theorem deals with the case that $(\la,\sigma)\in D$. In
this case there exist pointwise bounds for neither $f$ nor $g$.

\begin{theorem}\label{large-fg}
Suppose $\varepsilon>0$, $\alpha,\beta\in(2,n+2)$,
\begin{equation}\label{A1}
\lambda> \frac{n}{\a-2} \quad\text{and}\quad
\sigma>\frac{n+2-\a}{\beta-2}+\frac{n}{\beta-2}\frac{1}{\lambda}.
\end{equation}
Let $h:(0,1)\to (0,\infty)$ be a continuous function such that
$\lim_{t\to 0^+} h(t)=\infty$. Then there exist solutions
\begin{equation}\label{A2}
f,g\in C^\infty(\R^n\setminus\{0\})\cap L^1(\R^n)
\end{equation}
of the system \eqref{op2}
such that
\begin{equation}\label{A3}
f(x)\not={\cal O}(h(|x|))\quad\text{as }x\to 0
\end{equation}
and
\begin{equation}\label{A4}
g(x)\not={\cal O}(h(|x|))\quad\text{as }x\to 0.
\end{equation}
\end{theorem}

\section{Preliminary results}\label{prem}

\subsection{Nonlinear Riesz potentials}\label{riesz}
Let $B$ be a ball in $\R^n$, $n\geq 3,$ and $f\in L^\infty(B)$ be a nonnegative function. For any $a\in (0,n)$ we define  the Riesz potential ${\mathbf I}_a f$ of order $a$ by
$$
{\mathbf I}_a f(x)=\intl_{B} \frac{f(y)}{|x-y|^{n-a}}dy \quad\mbox{ for all }x\in B.
$$
We also set
$$
{\mathbf U}_{a,b,\sigma} f:={\mathbf I_a}\left(({\mathbf I}_b f)^\sigma\right),
$$
where $a,b\in (0,n)$. If $a=b$ then ${\mathbf U}_{a,a,\sigma}f$ is the Havin-Maz'ya potential \cite{HM1972}.

\begin{prop}\label{cor_nonlin}
  Let $a,b\in (0,n)$ and $\sigma>\frac{a}{n-b}$. Then there exists a
  constant $C=C(n,\sigma, a,b)>0$ such that $$ \| {\mathbf
    U}_{a,b,\sigma} f\|_\infty\leq C\|f\|_1^{\frac{a+b\sigma}{n}}
  \|f\|_\infty^{\frac{\sigma(n-b)-a}{n}}\quad\mbox{ for all }f\in
  L^\infty(B)\,,f\geq 0.
$$
\end{prop}

\begin{proof} Let us first recall Hedberg's inequality \cite{H1972}
\neweq{hedberg}
\|{\mathbf I}_\gamma f\|_\infty\leq C(n,\gamma,p) \|f\|_p^{\frac{\gamma p}{n}}\|f\|_\infty^{1-\frac{\gamma p}{n}}\quad\mbox{ for all }0<\gamma<n\,,\; 1\leq p<\frac{n}{\gamma}.
\endeq

Let $g=({\mathbf I}_b f)^\sigma$. Using \eq{hedberg} we have
\neweq{hedb1}
\|{\mathbf I}_a g\|_\infty\leq C(n,a,p) \|g\|_p^{\frac{a p}{n}}\|g\|_\infty^{1-\frac{a p}{n}}\quad\mbox{ for all } \; 1\leq p<\frac{n}{a}
\endeq
and
$$
\|{\mathbf I}_b f\|_\infty\leq C(n,b) \|f\|_1^{\frac{b}{n}}\|f\|_\infty^{1-\frac{b}{n}}.
$$
This last estimate implies
\neweq{hedb3}
\|g\|_\infty= \|{\mathbf I}_b f\|_\infty^\sigma \leq C \|f\|_1^{\frac{\sigma b}{n}}\|f\|_\infty^{\sigma(1-\frac{b}{n})}.
\endeq
Since $\sigma>\frac{a}{n-b}$, we can find $s\in(1,n/b)$ and
$p\in(1,n/a)$ such that
\neweq{hedb2}
p\sigma=\frac{ns}{n-bs}.
\endeq
By standard Riesz potential estimates (see \cite[Lemma 7.12]{GT1983}) we have
\neweq{hedb4}
\|g\|_p=\|{\mathbf I}_b f\|^\sigma_{p\sigma} \leq C \|f\|_s^\sigma.
\endeq
We now use \eq{hedb3} and \eq{hedb4} in \eq{hedb1} to deduce
\neweq{hedb5}
\|{\mathbf U}_{a,b,\sigma} f\|_\infty=\|{\mathbf I}_a g\|_\infty\leq C \|f\|_s^{\frac{a p\sigma}{n}} \|f\|_1^{\frac{\sigma b}{n}(1-\frac{ap}{n})}\|f\|_\infty^{\sigma(1-\frac{b}{n})(1-\frac{a p}{n})}.
\endeq
Finally, using the estimate
$$
\|f\|_s\leq \|f\|_1^{\frac{1}{s}}\|f\|_\infty^{\frac{s-1}{s}}
$$
in \eq{hedb5} we obtain
$$
\begin{aligned}
\|{\mathbf U}_{a,b,\sigma} f\|_\infty&\leq C \|f\|_1^{\frac{a p\sigma}{ns}+ \frac{\sigma b}{n}(1-\frac{ap}{n})}\|f\|_\infty^{\frac{a p\sigma}{n}\frac{s-1}{s}+\sigma(1-\frac{b}{n})(1-\frac{a p}{n})}\\
&= C \|f\|_1^{\frac{a+b\sigma}{n}} \|f\|_\infty^{\frac{\sigma(n-b)-a}{n}}
\end{aligned}
$$
by \eq{hedb2}.
\end{proof}
\subsection{Further estimates}\label{intg}
In this part we collect some results which will be used in our proofs. A very important tool in our approach is the following result due to Brezis and Lions which presents a representation formula for nonnegative superharmonic functions in a punctured ball of $\R^n$.

\begin{lemma}\label{brezislions}$($see \cite{BL1981}$)$
Let $u$ be a $C^2$ nonnegative superharmonic function in $B_{2r}(0)\setminus\{0\}\subset \R^n$, $n\geq 3$, for some $r>0$. Then
$$
u,-\Delta u\in L^1(B_r(0))
$$
and there exist $m\geq 0$, $c=c(n)>0$ and a bounded harmonic function $h:B_r(0)\to \R$ such that
$$
u(x)=m|x|^{2-n}+c\intl_{|y|<r}\frac{-\Delta
u(y)}{|x-y|^{n-2}}dy+h(x) \quad\mbox{ for }0<|x|<r.
$$
\end{lemma}

\begin{lemma}\label{optimal} $($see \cite[Lemma 5.1]{GTV2014}$)$
Let $\varphi:(0,1)\to (0,1)$ be a continuous function such that $\lim_{t\to 0^+}\varphi(t)=0$.
Let $\{x_j\}\subset\R^n$, $n\geq 3$, be a sequence satisfying
\neweq{ch3_gr5}
0<4|x_{j+1}|<|x_j|<\frac{1}{2}\quad\text{and}\quad
\sum_{j=1}^\infty \varphi(|x_j|)<\infty,
\endeq
and $\{r_j\}\subset \R$ be such that $0<r_j\leq |x_j|/2$.

Then there exist a positive constant $A=A(n)$ and a positive function $u\in C^\infty(\R^n\setminus\{0\})$ such that
\neweq{ch3_gr6}
0\leq -\Delta u\leq \frac{\varphi(|x_j|)}{r_j^n}\quad\mbox{ in }B_{r_j}(x_j),
\endeq
\neweq{ch3_gr7}
 -\Delta u=0\quad\mbox{ in }\R^n\setminus\Big(\{0\}\cup \bigcup_{j=1}^\infty B_{r_j}(x_j)\Big),
\endeq
\neweq{ch3_gr8}
u\geq 1\quad\mbox{ in }\R^n\setminus\{0\},
\endeq
and
\neweq{ch3_gr9}
 u\geq \frac{A\varphi(|x_j|)}{r_j^{n-2}} \quad\mbox{ in } B_{r_j}(x_j).
\endeq
\end{lemma}

\begin{lemma}\label{pastgeneral}
Let $\alpha\in (0,n)$ and $v\in L^1(B_1(0))$ be a nonnegative function such that
$$
v(x)={\mathcal O}(|x|^{-\gamma})\quad\mbox{ as }x\to 0
$$
for some $\gamma\geq 0$. Then
\neweq{pastgeneralestimate}
\intl_{|y|<1}\frac{v(y)\,dy}{|x-y|^{\alpha}} ={\mathcal
O}(|x|^{-\alpha})+o\Big(|x|^{-\frac{\gamma
\alpha}{n}}\Big)\quad\mbox{ as }x\to 0.
\endeq
\end{lemma}
\begin{proof}
Choose $C>0$ and $R\in(0,1/4)$ such that
$$
v(y)\leq C|y|^{-\gamma} \quad \mbox{ for }0<|y|<2R.
$$
Let $\{x_j\}\subset B_R(0)\setminus\{0\}$ be a sequence which
converges to 0. Then
\neweq{z1estimate}
v(y)\leq C|x_j|^{-\gamma}\quad\mbox{ for }|y-x_j|<\frac{|x_j|}{2}
\endeq
and it suffices to prove the estimate \eq{pastgeneralestimate} with
$x$ replaced with $x_j$.

Define $r_j\in(0,|x_j|/2)$ by
\neweq{z2estimate}
\intl_{|y-x_j|<r_j} C|x_j|^{-\gamma}
dy=\intl_{|y-x_j|<\frac{|x_j|}{2}} v(y)dy\to 0\quad\mbox{ as }j\to
\infty.
\endeq
Then $r_j=o(|x_j|^{\gamma/n})$ as $j\to \infty$. Also, using
\eq{z1estimate} and \eq{z2estimate} we have
$$
\begin{aligned}
\intl_{|y-x_j|<r_j}\frac{C|x_j|^{-\gamma}-v(y)}{|y-x_j|^\alpha}dy
&\geq
\intl_{|y-x_j|<r_j}\frac{C|x_j|^{-\gamma}-v(y)}{r_j^\alpha}dy\\
&=\intl_{r_j<|y-x_j|<\frac{|x_j|}{2}}\frac{v(y)}{r_j^\alpha}dy\\
&\geq
\intl_{r_j<|y-x_j|<\frac{|x_j|}{2}}\frac{v(y)}{|y-x_j|^\alpha}dy
\end{aligned}
$$
which yields
$$
\intl_{|y-x_j|<\frac{|x_j|}{2}}\frac{v(y) dy}{|y-x_j|^\alpha}\leq
\intl_{|y-x_j|<r_j}\frac{C|x_j|^{-\gamma} dy}{|y-x_j|^\alpha}.
$$
Using this last estimate, for $j$ large we have
$$
\begin{aligned}
\intl_{|y|<1}\frac{v(y)\,dy}{|x_j-y|^{\alpha}}
&=\intl_{|y|<1, |y-x_j|>\frac{|x_j|}{2}}\frac{v(y)\,dy}{|y-x_j|^{\alpha}}+\intl_{|y-x_j|<\frac{|x_j|}{2}}\frac{v(y)\,dy}{|y-x_j|^{\alpha}}\\
&\leq C |x_j|^{-\alpha}+\intl_{|y-x_j|<\frac{|x_j|}{2}}\frac{v(y)\,dy}{|y-x_j|^{\alpha}}\\
&\leq C \left[ |x_j|^{-\alpha}+\intl_{|y-x_j|<r_j}\frac{|x_j|^{-\gamma}\,dy}{|y-x_j|^{\alpha}}\right]\\
&\le C \left[ |x_j|^{-\alpha}+r_j^{n-\alpha}|x_j|^{-\gamma} \right]\\
&={\cal O}(|x_j|^{-\alpha})+o\Big(|x_j|^{-\frac{\gamma \alpha}{n}}\Big)\quad\mbox{ as }j\to \infty.
\end{aligned}
$$
\end{proof}

\begin{corollary}\label{past}
Let $u$ be a $C^2$ nonnegative function in
$B_r(0)\setminus\{0\}\subset \R^n$, $n\geq 3$, $r>0$ such that for
some $\gamma\ge 0$ we have
$$
0\leq -\Delta u\leq C|x|^{-\gamma}\quad\mbox{ for } 0<|x|<r.
$$
Then
$$
u(x)={\mathcal O}(|x|^{2-n})
+o\Big(|x|^{-\frac{\gamma(n-2)}{n}}\Big)\quad\mbox{ as }x\to 0.
$$
\end{corollary}

\begin{proof}
  We apply the representation formula in Lemma \ref{brezislions} and
  then Lemma \ref{pastgeneral} with $v=-\Delta u$ and $\alpha=n-2$.
\end{proof}

\begin{lemma}\label{int1}
Let $\alpha, \beta <n$. Then there exists a constant $C=C(n,\alpha,\beta)>0$ such that
$$
\intl_{|y|<2}\frac{dy}{|x-y|^{\beta}|y|^{\alpha}}\leq
\left\{
\begin{aligned}
&\frac{C}{|x|^{\alpha+\beta-n}}&\quad \mbox{ if }\alpha+\beta>n,\\
&C\ln\frac{2}{|x|}&\quad \mbox{ if }\alpha+\beta=n,\\
&C &\quad \mbox{ if }\alpha+\beta<n.
\end{aligned}
\right. \quad\mbox{ for }0<|x|<1.
$$
\end{lemma}
\begin{proof} We could use the convolution formula (see Stein \cite[pg. 118]{stein})
\neweq{stein}
\intl_{\R^n} \frac{dy}{|x-y|^{\beta}|y|^{\alpha}}=\frac{C(n,\alpha, \beta)}{|x|^{\alpha+\beta-n}}\quad\mbox{ for all }x\in \R^n,
\endeq
which holds whenever $\alpha+\beta>n$.
However, we shall give here a direct and simpler proof.

Let $x\in B_1(0)\setminus\{0\}$ and $r=|x|$. Under the change of
variable $x=r\xi$, $y=r\eta$,  we have $|\xi|=1$ and
$$
\begin{aligned}
\intl_{|y|<2}\frac{dy}{|x-y|^{\beta}|y|^{\alpha}}&=r^{n-\alpha-\beta}\intl_{|\eta|<\frac{2}{r}} \frac{d\eta}{|\xi-\eta|^{\beta}|\eta|^{\alpha}}\\
&=r^{n-\alpha-\beta}\left[ \intl_{|\eta|<2} \frac{d\eta}{|\xi-\eta|^{\beta}|\eta|^{\alpha}}+
\intl_{2<|\eta|<\frac{2}{r}} \frac{d\eta}{|\xi-\eta|^{\beta}|\eta|^{\alpha}}\right]\\
&= r^{n-\alpha-\beta}\left[ C(n,\alpha,\beta)+
\intl_{2<|\eta|<\frac{2}{r}} \frac{d\eta}{|\xi-\eta|^{\beta}|\eta|^{\alpha}}\right]\\
&\leq Cr^{n-\alpha-\beta}
\left[1+\intl_{2<|\eta|<\frac{2}{r}} \frac{d\eta}{|\eta|^{\alpha+\beta}}\right]\\
&\leq \left\{
\begin{aligned}
&\frac{C}{r^{\alpha+\beta-n}} &\quad \mbox{ if }\alpha+\beta>n,\\
&C\ln\frac{2}{r}&\quad \mbox{ if }\alpha+\beta=n,\\
&C &\quad \mbox{ if }\alpha+\beta<n.
\end{aligned}
\right.
\end{aligned}
$$
 \end{proof}

\begin{corollary}\label{corstein}
Let $\alpha,\beta <n$ and $R>0$. Then there exists a constant $C=C(n,\alpha,\beta)>0$ such that
for all $x,z\in B_R(0)$, $x\neq z$ we have
\neweq{corstein1}
\intl_{|y|<R}\frac{dy}{|x-y|^{\beta}|y-z|^{\alpha}}\leq
\left\{
\begin{aligned}
&\frac{C}{|x-z|^{\alpha+\beta-n}}&\quad \mbox{ if }\alpha+\beta>n,\\
&C\ln\frac{4R}{|x-z|}&\quad \mbox{ if }\alpha+\beta=n,\\
&\frac{C}{R^{\alpha+\beta-n}} &\quad \mbox{ if }\alpha+\beta<n.
\end{aligned}
\right.
\endeq
\end{corollary}
\begin{proof}
Under the change of variables
$\xi=\frac{x-z}{2R}$, $\eta=\frac{y-z}{2R}$ we find
$\xi\in B_1(0)\setminus\{0\}$ and thus by Lemma \ref{int1} we have
$$
\begin{aligned}
\intl_{|y|<R}\frac{dy}{|x-y|^{\beta}|y-z|^{\alpha}}& \le (2R)^{n-\alpha-\beta}\intl_{|\eta|<2}\frac{d\eta}{|\xi-\eta|^{\beta}|\eta|^{\alpha}}\\
& \leq \left\{
\begin{aligned}
&\frac{C}{(2R|\xi|)^{\alpha+\beta-n}}&\quad \mbox{ if }\alpha+\beta>n,\\
&C\ln\frac{2}{|\xi|}&\quad \mbox{ if }\alpha+\beta=n,\\
&\frac{C}{(2R)^{\alpha+\beta-n}} &\quad \mbox{ if }\alpha+\beta<n.
\end{aligned}
\right.
\end{aligned}
$$
This clearly implies \eq{corstein1}.
\end{proof}

\begin{lemma}\label{int2}
Let $\sigma>0$ and $\gamma\in (0,n)$. There exists a constant $C=C(n,\sigma, \gamma)>0$ such that
$$
\intl_{|y|<1}\frac{\ln^\sigma \frac{4}{|y-z|}}{|x-y|^{\gamma}}dy\leq C\quad\mbox{ for all }x,z\in B_1(0).
$$
\end{lemma}
\begin{proof}
This follows from Riesz potential estimates (see \cite[Lemma 7.12]{GT1983}).
\end{proof}

\begin{lemma}\label{uv}
Suppose $u$ and $v$ are  $C^2(\R^n\setminus\{0\})\cap L^1(\R^n)$ positive
solutions of \eq{nonloc} where $\lambda,\sigma\ge 0$ and
$\alpha,\beta\in (0,n)$. Then
\begin{equation}\label{lap-uv-L1}
-\Delta u,-\Delta v\in L^1(B_1(0))
\end{equation}
and for some positive constant $C$ we have
\begin{equation}\label{nonloc1}
\left\{
\begin{aligned}
&0\leq -\Delta u(x)\leq C\left( \intl_{|y|<1} \frac{v(y)\,dy}{|x-y|^{\alpha}} \right) ^\lambda \\
&0\leq -\Delta v(x)\leq C\left(\intl_{|y|<1} \frac{u(y)\,dy}{|x-y|^{\beta}} \right) ^\sigma
\end{aligned}
\quad \mbox{ for } 0<|x|<1,
\right.
\end{equation}
\begin{equation}\label{B-L}
\left\{
\begin{aligned}
&u(x)\leq C\left(|x|^{2-n}+\intl_{|y|<1}\frac{-\Delta
    u(y)\,dy}{|x-y|^{n-2}} \right) \\
&v(x)\leq C\left(|x|^{2-n}+\intl_{|y|<1}\frac{-\Delta
    v(y)\,dy}{|x-y|^{n-2}} \right)
\end{aligned}
\quad \mbox{ for } 0<|x|<1,
\right.
\end{equation}
and
\begin{equation}\label{combine}
\left\{
\begin{aligned}
-\Delta u(x)&\leq C\left[\int\limits_{|y|<1} \frac{dy}{|y-x|^{\alpha} |y|^{n-2}}+\int\limits_{|z|<1} -\Delta v(z)\left(\int\limits_{|y|<1} \frac{dy}{|x-y|^{\alpha}|y-z|^{n-2}}\right)dz \right]^\lambda\\
-\Delta v(x)&\leq C\left[\int\limits_{|y|<1} \frac{dy}{|y-x|^{\beta} |y|^{n-2}}+\int\limits_{|z|<1} -\Delta u(z)\left(\int\limits_{|y|<1} \frac{dy}{|x-y|^{\beta}|y-z|^{n-2}}\right)dz \right]^\sigma
\end{aligned}
\right.
\end{equation}
for $0<|x|<1$.
\end{lemma}

\begin{proof}
Lemma \ref{brezislions} implies \eqref{lap-uv-L1} holds.
Since
\[
\intl_{|y|<1}\frac{v(y)\,dy}{|x-y|^\alpha}
>\intl_{|y|<1}\frac{v(y)\,dy}{2^\alpha}=:C_1>0\quad\text{for }|x|<1
\]
and
\[
\intl_{|y|>1}\frac{v(y)\,dy}{|x-y|^\alpha}
\le \intl_{1<|y|<2}\frac{\max_{1\le|y|\le 2}v(y)}{|x-y|^\alpha}\,dy
+\intl_{|y|>2}v(y)\,dy\le C_2
\quad\text{for }|x|<1
\]
we see that
\[
\intl_{\R^n}\frac{v(y)\,dy}{|x-y|^\alpha}
\le \left(1+\frac{C_2}{C_1}\right)\intl_{|y|<1}\frac{v(y)\,dy}{|x-y|^\alpha}
\quad\text{for }|x|<1.
\]
Thus the first line of \eqref{nonloc1} follows from
\eqref{nonloc}. The second line of \eqref{nonloc1} is proved
similarly.

Inequalities \eqref{B-L} follow from Lemma \ref{brezislions}.
Substituting \eqref{B-L} in \eqref{nonloc1} we get \eqref{combine}.
\end{proof}

\section{Proof of Theorem \ref{thm1}}\label{sec-thm1}

By Lemma \ref{uv}, $u$ and $v$ satisfy \eqref{lap-uv-L1}--\eqref{combine}.

If $\beta<2$ then \eqref{combine}, \eq{lap-uv-L1}, Lemma \ref{int1}
and Corollary \ref{corstein} yield
$$
-\Delta v(x)\leq C\left[1+\intl_{|z|<1}-\Delta
  u(z)dz\right]^\sigma\leq C\quad\mbox{ for } 0<|x|<1
$$
and from Corollary \ref{past} we find $v(x)$ satisfies \eqref{thm1estv}.

Assume next that $\beta=2$. Then using Lemma \ref{int1} and Corollary
\ref{corstein} we obtain from \eq{combine} that
\begin{equation}\label{new1}
-\Delta v(x)\leq C\left[\ln \frac{4}{|x|}+ \int\limits_{|z|<1}
  \Big(\ln\frac{4}{|x-z|}\Big)  (-\Delta u(z))\, dz
\right]^\sigma\quad\text{for }0<|x|<1.
\end{equation}
Since increasing $\sigma$ increases the right side of \eqref{new1}, it
follows from \eqref{new1} that there exists $\gamma>1$ such that $u$
and $v$ satisfy
\[
-\Delta v(x)\leq C\left[\ln \frac{4}{|x|}+ \int\limits_{|z|<1}
  \Big(\ln\frac{4}{|x-z|}\Big)  (-\Delta u(z))\, dz
\right]^\gamma\quad\text{for }0<|x|<1.
\]
Thus by Jensen's inequality, we have
$$
\begin{aligned}
-\Delta v(x)&\leq C\left[\ln^\gamma \frac{2}{|x|}+ \left(\int\limits_{|z|<1} \Big(\ln\frac{4}{|x-z|}\Big)  (-\Delta u(z)) dz \right)^\gamma\right]\\
&= C\left[\ln^\gamma \frac{2}{|x|}+ \|\Delta u\|_{L^1(B_1(0))}^\gamma \left(\int\limits_{|z|<1} \Big(\ln\frac{4}{|x-z|}\Big)  \frac{-\Delta u(z)}{\|\Delta u\|_{L^1(B_1(0))}  } dz \right)^\gamma\right]\\
&\leq C\left[\ln^\gamma \frac{2}{|x|}+  \int\limits_{|z|<1}
  \Big(\ln^\gamma \frac{4}{|x-z|}\Big)  (-\Delta u(z))  dz\right]
\quad\mbox{ for } 0<|x|<1.
\end{aligned}
$$
This last estimate combined with \eq{B-L}, \eqref{lap-uv-L1}, and Lemma
\ref{int2}
yields
$$
\begin{aligned}
v(x)& \leq C\left[ |x|^{2-n}+\intl_{|y|<1}\frac{-\Delta v(y)}{|x-y|^{n-2}}dy\right]\\
&\leq C\left[|x|^{2-n}+\intl_{|y|<1} \frac{\ln^\gamma \frac{2}{|y|}}{|x-y|^{n-2}}dy +
\intl_{|z|<1}-\Delta u(z)\left(\intl_{|y|<1} \frac{\ln^\gamma \frac{4}{|y-z|}}{|x-y|^{n-2}}dy\right) dz  \right]\\
&\leq C\Big[|x|^{2-n}+\|\Delta u\|_{L^1(B_1(0))}\Big]\\
&\leq C|x|^{2-n} \quad\mbox{ for }0<|x|<1.
\end{aligned}
$$
We have thus established \eq{thm1estv} for $\beta\leq 2$. Now, from the first equation of \eq{nonloc1} and Lemma \ref{int1} we find
$$
\begin{aligned}
-\Delta u(x)&\leq C\left(\int\limits_{|y|<1} \frac{v(y)\,dy}{|x-y|^{\alpha}}\right)^\lambda\leq C\left(  \int\limits_{|y|<1}  \frac{dy}{|x-y|^{\alpha}|y|^{n-2}}\right)^\lambda\\
& \leq \left\{
\begin{aligned}
& C |x|^{-\lambda(\alpha-2)}&\quad \mbox{ if }\alpha>2,\\
&C\ln^\lambda\frac{2}{|x|}&\quad \mbox{ if }\alpha=2,\\
&C &\quad \mbox{ if }\alpha<2.
\end{aligned}
\right. \quad\mbox{ for }0<|x|<1.
\end{aligned}
$$
If $\alpha\leq 2$ we use \eq{B-L} and Lemma \ref{int2} to deduce $u(x)={\mathcal O}(|x|^{2-n})$ as $x\to 0$. If $\alpha>2$ then we apply directly Corollary \ref{past} to derive
$$
u(x)={\mathcal O}(|x|^{2-n})+o\Big(|x|^{-\frac{\lambda(\alpha-2)(n-2)}{n}}\Big)\quad\mbox{ as }x\to 0,
$$
and complete the proof of \eq{thm1estu}.

\section{Proof of Theorem \ref{optimalthm1}}

Define $\varphi:(0,1)\to (0,1)$ by $\varphi=\sqrt{h}$. Let
$\{x_j\}\subset \R^n$ be a sequence satisfying \eq{ch3_gr5} and
\begin{equation}\label{radius}
r_j=|x_j|^{\frac{\lambda(\alpha-2)}{n}}<<|x_j|\quad\mbox{ as }j\to \infty.
\end{equation}
By Lemma \ref{optimal} there exist a positive constant $A=A(n)$ and a
positive function $u\in C^\infty(\R^n\setminus\{0\})$ that
satisfies \eq{ch3_gr6}--\eq{ch3_gr9}.
In particular, $-\Delta u\ge 0$ in $\R^n\setminus\{0\}$.
Let
\[
\hat u=u\chi+w(1-\chi)\quad\text{and}\quad v=|x|^{-(n-2)}\chi+w(1-\chi)
\]
where $w\in C^\infty(\R^n)\cap L^1(\R^n)$ is a positive function and
$\chi\in C^\infty(\R^n)$ is a nonnegative function satisfying $\chi=1$ in
$B_2(0)$ and $\chi=0$ in $\R^n\setminus B_3(0)$. Then $\hat u, v\in
C^\infty(\R^n\setminus\{0\})\cap L^1(\R^n)$ by Lemma \ref{brezislions}. Also, $\hat
u=u$ and $v=|x|^{-(n-2)}$ in $B_2(0)\setminus\{0\}$. For simplicity of
notation, we again denote $\hat u$ by $u$.  Since $v$ is harmonic in
$B_2(0)\setminus\{0\}$ we only need to check that $u$ and $v$ satisfy
\neweq{remain} 0\leq -\Delta u(x)\leq\left(\;\int\limits_{\R^n}
  \frac{v(y)\,dy}{|x-y|^{\alpha}}\right)^\lambda
\quad\mbox{ for }0<|x|<2
\endeq
and that \eq{optimalesti1} holds. In fact, owing to \eq{ch3_gr7}, we only need to check that \eq{remain} is valid in $\bigcup_{j=1}^\infty B_{r_j}(x_j)$.

For $x\in B_{r_j}(x_j)$ we have $x\in B_1(0)$ and
\neweq{remain1}
\int\limits_{\R^n} \frac{v(y)\,dy}{|x-y|^{\alpha}}\geq \!\!\!\!\!
\!\!\int\limits_{|y-x|<\frac{|x|}{2}}
\frac{dy}{|x-y|^{\alpha}|y|^{n-2}} \geq
\frac{C}{|x|^{n-2}}\int\limits_{|y-x|<\frac{|x|}{2}}
\frac{dy}{|x-y|^{\alpha}}=\frac{C}{|x|^{\alpha-2}}
>\frac{C}{|x_j|^{\a-2}}.
\endeq
We now combine \eq{ch3_gr6}, \eq{radius} and \eq{remain1} to obtain
$$
-\Delta u(x)\leq \frac{\varphi(|x_j|)}{r_j^n}\leq \frac{1}{r_j^n}= \frac{1}{|x_j|^{\lambda(\alpha-2)}}\leq
C\left(\;\int\limits_{\R^n} \frac{v(y)\,dy}{|x-y|^{\alpha}}\right)^\lambda
$$
for all $x\in B_{r_j}(x_j)$. This establishes \eq{remain}. To check \eq{optimalesti1} we use \eq{ch3_gr9}, \eq{radius}  and obtain
$$
\begin{aligned}
\frac{u(x_j)}{h(|x_j|)|x_j|^{-\frac{\lambda(\alpha-2)(n-2)}{n}}} &\geq \frac{A\varphi(|x_j|)}{h(|x_j|)r_j^{n-2}|x_j|^{-\frac{\lambda(\alpha-2)(n-2)}{n}}}\\
&= \frac{A}{\sqrt{h(|x_j|)}}\to \infty\quad\mbox{ as }j\to\infty.
\end{aligned}
$$

\section{Proof of Theorems \ref{thm3}--\ref{thm4B},
  \ref{crucial1}, and \ref{crucial2}}

The theorems in the title of this section are either immediate
consequences of the following theorem or follow very easily from
it. Its proof is the crux of this paper. Specifically, estimates
\eqref{f-est} and \eqref{g-est} immediately give Theorems
\ref{crucial1} and \ref{crucial2}, and, as we will see at the end of this
section, Theorems \ref{thm3}--\ref{thm4B} follow easily from
estimates \eqref{e17} and \eqref{e18}.

\begin{theorem}\label{crucial}
Assume $\alpha,\beta\in(2,n+2)$, $\lambda\ge 0$, and
\neweq{lasi}
0\le\sigma<\min\left\{\frac{n}{\beta-2},
\frac{n+2-\alpha}{\beta-2}+\frac{n}{\beta-2}\frac{1}{\lambda}\right\}.
\endeq
Let $f$ and $g$ be $L^1(B_1(0))$ solutions
of \eqref{hls_syst} where $M$ is a positive constant. Then
\begin{equation}\label{f-est}
f(x)={\mathcal O}\left(|x|^{-\lambda(\alpha-2)}\right)\quad\mbox{ as }x\to 0
\end{equation}
and
\begin{equation}\label{g-est}
g(x)={\mathcal O}\left(|x|^{-\sigma(\beta-2)}\right)
+o\left(|x|^{-\frac{\lambda(\alpha-2)\sigma(\beta-2)}{n}}\right)
\quad\mbox{ as }x\to 0.
\end{equation}
Also, for $2<s<n+2$ we have
\neweq{e17}
 \intl_{|y|<1}\frac{f(y)\,dy}{|x-y|^{s-2}}={\mathcal O}\left(|x|^{-(s-2)}\right)+o\Big(|x|^{-\frac{\lambda(\alpha-2)(s-2)}{n}}\Big) \quad\mbox{ as }x\to 0
\endeq
and
\neweq{e18}
\intl_{|y|<1}\frac{g(y)\,dy}{|x-y|^{s-2}} ={\mathcal O}\left(|x|^{-(s-2)}\right)+o\Big(|x|^{-\frac{\lambda(\alpha-2)[\sigma(\beta-2)-(n+2-s)]}{n}}\Big)  \quad\mbox{ as }x\to 0.
\endeq
In particular,
\begin{equation}\label{e19.5}
\intl_{|y|<1}\frac{g(y)\,dy}{|x-y|^{\a-2}} ={\mathcal O}\left(|x|^{-(\a-2)}\right)
\quad\mbox{ as }x\to 0.
\end{equation}
\end{theorem}

\begin{proof}
The estimate \eqref{e19.5} follows from \eqref{e18} and
\eqref{lasi}. Also, using Lemma \ref{pastgeneral}, we see that
\eqref{f-est} implies \eqref{e17}. Moreover, \eqref{e17} with
$s=\b$ combined with \eqref{hls_syst} implies \eqref{g-est}. Hence it
remains only to prove \eqref{f-est} and \eqref{e18}.

We first prove \eqref{f-est}.
If $\lambda=0$ then \eqref{f-est} follows immediately from
\eqref{hls_syst}. Hence we can assume for the proof of  \eqref{f-est}
that
\begin{equation}\label{la}
\lambda>0.
\end{equation}
Moreover, since the estimate \eqref{f-est} for $f$
does not depend on $\sigma$ and since increasing $\sigma$ weakens the
conditions on $f$ and $g$ in the system \eqref{hls_syst}, we can also
assume for the proof of \eqref{f-est} that
\begin{equation}\label{sigma-bound}
\sigma>\frac{n+2-\alpha}{\beta-2}.
\end{equation}
We divide the  proof of \eqref{f-est} into two steps.

\noindent{\bf Step 1:} For some $\gamma>n$ we have
\neweq{gammaexp}
f(x)={\mathcal O}(|x|^{-\gamma})\quad\mbox{ as }x\to 0.
\endeq

Let $\{x_j\}\subset\R^n$ be a sequence such that
\neweq{subseq}
0<4|x_{j+1}|<|x_j|<\frac{1}{2}\quad\mbox{ for }j=1,2,\dots.
\endeq
To prove \eq{gammaexp}, it suffices to prove
\neweq{gammaexp-seq}
f(x_j)={\mathcal O}(|x_j|^{-\gamma})\quad\mbox{ as }j\to \infty.
\endeq

Since
\[
\intl_{|y-x_j|>|x_j|/2,|y|<1}\frac{g(y)\,dy}{|x-y|^{\alpha-2}}
\le \left(\frac{4}{|x_j|}\right)^{\alpha-2}\intl_{|y|<1}g(y)\,dy
\le C|x_j|^{2-\alpha} \quad\mbox{ for }|x-x_j|<r_j:=\frac{|x_j|}{4},
\]
it follows from \eqref{hls_syst} that

\neweq{e3}
f(x)\leq C\left[ |x_j|^{2-\alpha}
+\intl_{|y-x_j|<\frac{|x_j|}{2}}\frac{g(y)\,dy}{|x-y|^{\alpha-2} } \right]^\lambda\quad\mbox{ for }|x-x_j|<r_j;
\endeq
and similarly
\neweq{e4}
g(x)\leq C\left[ |x_j|^{2-\beta}
+\intl_{|y-x_j|<\frac{|x_j|}{2}}\frac{f(y)\,dy}{|x-y|^{\beta-2} } \right]^\sigma\quad\mbox{ for }|x-x_j|<r_j.
\endeq
Let now $f_j,g_j:B_2(0)\to [0,\infty)$ be defined by
$$
f_j(\xi)=r_j^{n}f(x_j+r_j\xi)\,,\quad  g_j(\xi)=r_j^{n}g(x_j+r_j\xi).
$$
Since $f,g\in L^1(B_1(0))$ we have
\neweq{e5}
\|f_j\|_{L^1(B_2(0))}\to 0\,,\quad \|g_j\|_{L^1(B_2(0))}\to 0\quad\mbox{ as } j\to \infty.
\endeq
Further, with the change of variable $y=x_j+r_j\zeta$ in \eq{e3} and \eq{e4} we find
\neweq{e6}
r_j^{-n}f_j(\xi)=f(x_j+r_j\xi)\leq C|x_j|^{-\lambda(\alpha-2)}
\left[ 1+
\intl_{|\zeta|<2}\frac{g_j(\zeta)\,d\zeta}{|\xi-\zeta|^{\alpha-2} } \right]^\lambda\quad\mbox{ for }|\xi|<1,
\endeq
and
\neweq{e7}
r_j^{-n}g_j(\zeta)=g(x_j+r_j\zeta)\leq C|x_j|^{-\sigma(\beta-2)}
\left[ 1+
\intl_{|\eta|<2}\frac{f_j(\eta)\,d\eta}{|\zeta-\eta|^{\beta-2} } \right]^\sigma\quad\mbox{ for }|\zeta|<1.
\endeq

For any $a\in (0,n)$, $r>0$ and any $f\in L^1(B_r(0))$, $f\geq 0$ we denote by ${\bf I}_{a,r}f$ the Riesz potential
$$
{\bf I}_{a,r}f(x)=\intl_{B_r(0)}\frac{f(y)\,dy}{|x-y|^{n-a}}
$$
and we define
$$
{\bf U}_{a,b,\sigma; r} f:={\mathbf I_{a,r}}\left(({\mathbf I}_{b,r} f)^\sigma\right).
$$
Let $R\in (0,1/2]$. By \eqref{e5} we have
$$
\intl_{|\zeta|<2}\frac{g_j(\zeta)\,d\zeta}{|\xi-\zeta|^{\alpha-2} } \leq C\left[\frac{1}{R^{\alpha-2}}+\intl_{|\zeta|<2R}\frac{g_j(\zeta)\,d\zeta}{|\xi-\zeta|^{\alpha-2} } \right]\quad\mbox{ for }|\xi|<R.
$$
In other words,
\neweq{e8}
\intl_{|\zeta|<2}\frac{g_j(\zeta)\,d\zeta}{|\xi-\zeta|^{\alpha-2} } \leq C\left[\frac{1}{R^{\alpha-2}}+{\bf I}_{n+2-\alpha,2R}(g_j)(\xi)\right] \quad\mbox{ for }|\xi|<R.
\endeq
Similarly, we find
\neweq{e9}
\intl_{|\eta|<2}\frac{f_j(\eta)\,d\eta}{|\zeta-\eta|^{\beta-2} } \leq C\left[\frac{1}{R^{\beta-2}}+{\bf I}_{n+2-\beta,4R}(f_j)(\zeta)\right] \quad\mbox{ for }|\zeta|<2R.
\endeq
Combining \eq{e6}, \eq{e7}, \eq{e8} and \eq{e9} we deduce
\neweq{e10}
f_j(\xi)\leq Cr_j^{n-\lambda(\alpha-2)}\left\{  \frac{1}{R^{\lambda(\alpha-2)}}+\Big[{\bf I}_{n+2-\alpha,2R}(g_j)\Big]^\lambda(\xi)\right\} \quad\mbox{ for }|\xi|<R,
\endeq
\neweq{e11}
g_j(\zeta)\leq Cr_j^{n-\sigma(\beta-2)}\left\{  \frac{1}{R^{\sigma(\beta-2)}}+\Big[{\bf I}_{n+2-\beta,4R}(f_j)\Big]^\sigma(\zeta)\right\} \quad\mbox{ for }|\zeta|<2R.
\endeq
Now, from \eq{e11} we find for all $\xi\in \R^n$ that
$$
\begin{aligned}
{\bf I}_{n+2-\alpha,2R}(g_j)(\xi)&\leq Cr_j^{n-\sigma(\beta-2)}\left\{ R^{n-\alpha+2-\sigma(\beta-2)}+{\bf I}_{n+2-\alpha,4R}\Big[{\bf I}_{n+2-\beta,4R}(f_j)\Big]^\sigma(\xi)  \right\}\\
&=Cr_j^{n-\sigma(\beta-2)}\left\{ R^{n-\alpha+2-\sigma(\beta-2)}+{\bf U}_{n+2-\alpha,n+2-\beta,\sigma;4R}(f_j)(\xi)  \right\}.\\
\end{aligned}
$$
It therefore follows from \eqref{e10} that there exists a positive
constant $a$ which depends only on $n$, $\alpha$, $\beta$,
$\lambda$, and $\sigma$ such that
\neweq{e13}
f_j(\xi)\leq \frac{C}{(Rr_j)^{a}} \left\{ 1+\Big[{\bf V}(f_j)(\xi)  \Big]^\lambda\right\} \quad\mbox{ for }|\xi|<R\leq \frac{1}{2},
\endeq
where
$$
{\mathbf V}(f):={\bf U}_{n+2-\alpha,n+2-\beta,\sigma; 4R}(f).
$$
At this stage, to prove for some $\gamma>n$ that \eq{gammaexp-seq}
holds, it suffices to show that for some $\gamma>0$ the sequence
$\{r^\gamma_jf_j(0)\}$ is bounded.  This will be achieved by means
of the following auxiliary result.

\begin{lemma}\label{aux}
Suppose the sequence
\begin{equation}\label{s1}
\{r^\gamma_jf_j\} \quad\text{is bounded in } L^p(B_{4R}(0))
\end{equation}
for some constants $\gamma\ge 0$, $p\in [1,\infty)$, and $R\in (0,1/2]$.
Let $\delta=\gamma\lambda\sigma+a$ where $a$ is as in
\eqref{e13}. Then either the sequence
\begin{equation}\label{s2}
\{r^\delta_jf_j\} \quad\text{is bounded in } L^\infty(B_{R}(0))
\end{equation}
or there exists a positive constant
$C_0=C_0(n,\lambda,\sigma,\alpha,\beta)$ such that the sequence
\[
\{r^\delta_jf_j\} \quad\text{is bounded in } L^q(B_{R}(0))
\]
for some $q\in(p,\infty)$ satisfying
\begin{equation}\label{s3}
\frac{1}{p}-\frac{1}{q}>C_0.
\end{equation}
\end{lemma}

\begin{proof}[Proof of Lemma \ref{aux}]
It follows from \eqref{lasi} that there exists
$
\varepsilon=\varepsilon(n,\lambda,\sigma,\alpha,\beta)>0
$
such that
\begin{equation}\label{conditionepsilon}
\a,\b<n+2-\ep
\quad\text{and}\quad
0\le\sigma<\min\left\{\frac{n}{\beta-2+\varepsilon},
\frac{n+\lambda(n+2-\alpha-\varepsilon)}{\lambda(\beta-2+\varepsilon)}\right\}.
\end{equation}
By \eqref{e13} we have
\begin{equation}\label{eq55.4}
r_j^\delta f_j(\xi)\le\frac{C}{R^a}\left(1+(({\bf V}(r_j^\gamma
  f_j))(\xi))^\lambda\right) \quad\text{for } |\xi|<R.
\end{equation}
We can assume
\begin{equation}\label{eq55.5}
p\le n/(n+2-\beta)
\end{equation}
for otherwise from Riesz potential estimates (see
\cite[Lemma 7.12]{GT1983}) and
\eqref{s1} we find that the sequence $\{I_{n+2-\beta,4R}(r_j^\gamma f_j)\}$ is bounded in
$L^\infty (B_{4R}(0))$ and hence by \eqref{eq55.4} we see that
\eqref{s2} holds.

Define $p_1$ by
\begin{equation}\label{eq4.46}
  \frac{1}{p}-\frac{1}{p_1}=\frac{n+2-\beta-\varepsilon}{n},
 \end{equation}
where $\varepsilon$ is as in
\eqref{conditionepsilon}. By \eqref{eq55.5}, $p_1\in (p,\infty)$
and by Riesz potential estimates we have
\begin{equation}\label{eq4.47}
  \Vert({\bf I}_{n+2-\beta,4R}f_j)^\sigma \Vert_{p_1/\sigma}
=\Vert {\bf I}_{n+2-\beta,4R}f_j \Vert^{\sigma}_{p_1} \leq C\Vert f_j\Vert^{\sigma}_{p}
 \end{equation}
where $\Vert \cdot \Vert_p:=\Vert \cdot \Vert_{L^p (B_{4R}(0))}$.
Since, by \eqref{conditionepsilon},
\[
\frac{1}{p_1}=\frac{1}{p}-\frac{n+2-\beta-\varepsilon}{n}
\leq 1-\frac{n+2-\beta-\varepsilon}{n}
=\frac{\beta-2+\varepsilon}{n}
<\frac{1}{\sigma},
\]
we have
\begin{equation}\label{eq55.5.5}
p_1/\sigma>1.
\end{equation}

We can assume
\begin{equation}\label{eq55.6}
p_1/\sigma\le n/(n+2-\alpha)
\end{equation}
for otherwise by Riesz potential estimates and \eqref{eq4.47} we have
 $$
 \begin{aligned}
 \| {\mathbf V}(r_j^\gamma f_j)\|_\infty&=\|{\bf U}_{n+2-\alpha,n+2-\beta,\sigma;4R}
(r_j^\gamma f_j)\|_\infty\\
&\leq C\| ({\bf I}_{n+2-\beta,4R}(r_j^\gamma f_j))^\sigma\|_{\frac{p_1}{\sigma}}\\
&\leq C\| r_j^{\gamma} f_j\|^\sigma_{p}
 \end{aligned}
 $$
which is a bounded sequence by \eqref{s1}. Hence \eqref{eq55.4} implies
\eqref{s2}.

Define $p_2$ by
 \begin{equation}\label{e16}
 \frac{\sigma}{p_1}-\frac{1}{p_2}=\frac{n+2-\alpha-\varepsilon}{n}\quad\mbox{ and let }\quad q=\frac{p_2}{\lambda}.
\end{equation}
By \eqref{eq55.5.5} and \eqref{eq55.6}, $p_2\in(1,\infty)$ and by
Riesz potential estimates
$$
 \begin{aligned}
 \| {\mathbf V}(f_j)^\lambda \|_q
&=\|{\bf U}_{n+2-\alpha,n+2-\beta,\sigma;4R}(f_j)\|_{p_2}^\lambda\\
&\leq C\| ({\bf I}_{n+2-\beta,4R}f_j)^\sigma\|^\lambda_{\frac{p_1}{\sigma}}\\
&\leq C\| f_j\|^{\lambda\sigma}_{p} ,
 \end{aligned}
 $$
by \eqref{eq4.47}.
It follows therefore from \eqref{eq55.4} that
$$
\|r^{\delta}_jf_j\|_{L^q(B_R(0))}
\leq \frac{C}{R^a}\Big[1+\|r^\gamma_j f_j\|_{L^p(B_{4R}(0))}^{\lambda\sigma}\Big],
$$
which is a bounded sequence by \eqref{s1}.

It remains to prove that $q$ satisfies \eqref{s3} for some positive constant
$C_0=C_0(n,\lambda,\sigma,\alpha,\beta)$.
By \eq{eq4.46} and \eq{e16} we have
$$
\begin{aligned}
\frac{1}{p}-\frac{1}{q}&=\frac{1}{p}-\frac{\lambda}{p_2}=\frac{1}{p}-\lambda\Big[
\frac{\sigma}{p_1}-\frac{n+2-\alpha-\varepsilon}{n} \Big]\\
&= \frac{1}{p}+\frac{\lambda(n+2-\alpha-\varepsilon)}{n}-\frac{\lambda\sigma}{p_1}\\
&= \frac{1}{p}+\frac{\lambda(n+2-\alpha-\varepsilon)}{n}-\lambda\sigma\Big[\frac{1}{p}-\frac{n+2-\beta-\varepsilon}{n}\Big]\\
&=\frac{1-\lambda\sigma}{p}+\frac{\lambda(n+2-\alpha-\varepsilon)
+\lambda\sigma(n+2-\beta-\varepsilon)}{n}.
\end{aligned}
$$
If $\lambda\sigma\leq 1$ then
$$
\frac{1}{p}-\frac{1}{q}\geq \frac{\lambda(n+2-\alpha-\varepsilon)
+\lambda\sigma(n+2-\beta-\varepsilon)}{n}= C_1(n,\lambda,\sigma,\alpha,\beta)>0.
$$
by \eqref{conditionepsilon} and \eqref{la}.
If $\lambda\sigma>1$ then
$$
\begin{aligned}
\frac{1}{p}-\frac{1}{q}
&\geq 1-\lambda\sigma+\frac{\lambda(n+2-\alpha-\varepsilon)+\lambda\sigma(n+2-\beta-\varepsilon)}{n}\\
&=\frac{\lambda(\beta-2+\varepsilon)}{n}
\,\left[ \frac{n+\lambda(n+2-\alpha-\varepsilon)}{\lambda(\beta-2+\varepsilon)}-\sigma\right]\\
&= C_2(n,\lambda,\sigma,\alpha,\beta)>0
\end{aligned}
$$
by \eqref{conditionepsilon} and \eqref{la}. Thus \eqref{s3} holds with
$C_0=\min\{C_1,C_2\}$. This completes the proof of Lemma \ref{aux}.
 \end{proof}

 We are now ready to complete the proof of \eq{gammaexp}.  By
 \eqref{e5}, the sequence $\{f_j\}$ is bounded in $L^1(B_2(0))$.
 Starting with this fact and iterating Lemma \ref{aux} a finite number
 of times ($m$ times is enough if $m>1/C_0$) we see that there exists
 $R_0 \in(0,\frac{1}{2})$ and $\gamma>n$ such that sequence
 $\{r_j^\gamma f_j\}$ is bounded in $L^\infty(B_{R_0}(0))$. In
 particular $\{r_j^\gamma f_j(0)\}$ is a bounded sequence, whence
 \eq{gammaexp-seq} and \eqref{gammaexp}.

\medskip

 \noindent{\bf Step 2:} Proof of \eqref{f-est}.

Let $\{x_j\}\subset\R^n$ be a sequence satisfying
\eqref{subseq}. Then, as is Step 1, $f$ and $g$ satisfy \eqref{e3} and
\eqref{e4} where $r_j=|x_j|/4$.

By \eqref{gammaexp}, for some $\gamma>n$, we have
\begin{equation}\label{e20}
f(x)\le C|x_j|^{-\gamma}\quad\text{for } |x-x_j|<2r_j.
\end{equation}
Let
\[
\hat {\bf I}_{a,j}f(x)=\intl_{|y-x_j|<2r_j}\frac{f(y)\,
  dy}{|x-y|^{n-a}}\quad\text{and}\quad 
\hat {\bf U}_{a,b,\sigma,j}f=\hat {\bf I}_{a,j}((\hat {\bf I}_{b,j}f)^\sigma).
\]
Since, by \eqref{e4},
\begin{equation}\label{e19}
g(x)\leq C\left[|x_j|^{-\sigma(\beta-2)}+\left({\bf
      \hat I}_{n+2-\beta,j}f\right)^\sigma(x)  \right] \quad\mbox{ for }|x-x_j|<r_j,
\end{equation}
we find that
\begin{align}
\intl_{|y-x_j|<r_j}\frac{g(y)\,dy}{|x_j-y|^{\alpha-2}}
&\leq C\left[\intl_{|y-x_j|<r_j}\frac{|x_j|^{-\sigma(\beta-2)}}{|x_j-y|^{\alpha-2} } dy
+\intl_{|y-x_j|<2r_j}\frac{(\hat {\bf I}_{n+2-\beta,j}f(y))^\sigma}
{|x_j-y|^{\alpha-2} } dy\right]\notag\\
&\leq C\left[ |x_j|^{n+2-\alpha-\sigma(\beta-2)}
+\|\hat {\bf U}_{n+2-\alpha, n+2-\beta,\sigma,j}f\|_{L^\infty(B_{2r_j}(x_j))}
\right]\notag\\
&={\cal O}(|x_j|^{-(\alpha-2)})
+o\left(|x_j|^{-\frac{\gamma}{n}[\sigma(\beta-2)-(n+2-\alpha)]}\right)
\quad\text{as }j\to\infty\label{ee19}
\end{align}
where the big ``oh'' term follows from \eqref{crucial} and the little
``oh'' term follows from \eqref{e20}, $f\in L^1(B_1(0))$,
\eqref{sigma-bound}, and Proposition \ref{cor_nonlin}.

Since $g\in L^1(B_1(0))$ we have
\begin{align*}
\intl_{|y-x_j|<2r_j}\frac{g(y)\,dy}{|x_j-y|^{\alpha-2}}
&\leq \intl_{|y-x_j|<r_j}\frac{g(y)\,dy}{|x_j-y|^{\alpha-2}}
+\intl_{r_j<|y-x_j|<2r_j}\frac{g(y)}{r_j^{\alpha-2}}dy\\
&\leq \intl_{|y-x_j|<r_j}\frac{g(y)\,dy}{|x_j-y|^{\alpha-2}}
+o(|x_j|^{2-\alpha})\quad\text{as }j\to\infty.
\end{align*}
We therefore deduce from \eqref{e3} and \eqref{ee19} that
\[
f(x_j)={\cal O}(|x_j|^{-\lambda(\alpha-2)})
+o\left(|x_j|^{-\frac{\lambda\gamma}{n}[\sigma(\beta-2)-(n+2-\alpha)]}\right)
\quad\text{as }j\to\infty.
\]
Thus, since $\{x_j\}$ was an arbitrary sequence satisfying
\eqref{subseq}, we have
\begin{equation}\label{ee}
f(x)={\cal O}(|x|^{-\lambda(\alpha-2)})
+o\left(|x|^{-\frac{\lambda\gamma}{n}[\sigma(\beta-2)-(n+2-\alpha)]}\right)
\quad\text{as }x\to 0.
\end{equation}

Let $\{\gamma_j\}$ be a sequence of real numbers defined by
$\gamma_0=\gamma$ and
$$
\gamma_{j+1}=\frac{\lambda\gamma_j}{n}[\sigma(\beta-2)-(n+2-\alpha)]
\quad\mbox{ for } \, j=0,1,\dots.
$$
Since $\sigma$ satisfies \eq{lasi} and \eqref{sigma-bound} we have
$\{\gamma_j\}\subset (0,\infty)$ and $\gamma_j\to 0$ as $j\to
\infty$. Thus, iterating finitely many times the procedure of going
from \eqref{gammaexp} to \eqref{ee} we obtain \eqref{f-est}.  \medskip

We now prove \eqref{e18}. Since increasing $\sigma$ weakens the
conditions on $f$ and $g$ in the system \eqref{hls_syst} and since
increasing $\sigma$ to a value slightly larger than $(n+2-s)/(\beta-2)$
does not change the estimate  \eqref{e18}, we can assume for the proof
of  \eqref{e18} that
\begin{equation}\label{sigma-bound2}
\sigma>\frac{n+2-s}{\beta-2}.
\end{equation}
Let $\{x_j\}\subset\R^n$ be a sequence satisfying \eqref{subseq}. Then,
as before, $g$ satisfies \eqref{e19} where $r_j=|x_j|/4$.
Repeating the calculation \eqref{ee19}, except this time with
$\alpha=s$ and $\gamma=\lambda(\alpha-2)$ and using \eqref{sigma-bound2}
instead of \eqref{sigma-bound}, we get
\[
\intl_{|y-x_j|<r_j}\frac{g(y)\,dy}{|x_j-y|^{s-2} }
={\mathcal O}(|x_j|^{2-s})
+o\Big(|x_j|^{-\frac{\lambda(\alpha-2)[\sigma(\beta-2)-(n+2-s)]}{n}}\Big) \quad\mbox{ as }j\to \infty.
\]
Thus
\begin{align*}
\intl_{|y|<1}\frac{g(y)\,dy}{|x_j-y|^{s-2} }
&=\intl_{|y-x_j|<r_j}\frac{g(y)\,dy}{|x_j-y|^{s-2} }
+\intl_{|y-x_j|>r_j,|y|<1}\frac{g(y)\,dy}{|x_j-y|^{s-2} } \\
&\le C|x_j|^{2-s}+\intl_{|y-x_j|<r_j}\frac{g(y)\,dy}{|x_j-y|^{s-2} }\\
&={\mathcal O}(|x_j|^{2-s})+o\Big(|x_j|^{-\frac{\lambda(\alpha-2)[\sigma(\beta-2)-(n+2-s)]}{n}}\Big)  \quad\mbox{ as }j\to \infty
\end{align*}
which proves \eqref{e18}. This finishes the proof of Theorem \ref{crucial}.
\end{proof}

We are now able to easily prove Theorems \ref{thm3}--\ref{thm4B}.

\begin{proof}[Proof of Theorems \ref{thm3}--\ref{thm4B}]
  By Lemma \ref{uv}, $u$ and $v$ satisfy
  \eqref{lap-uv-L1}--\eqref{combine}. Let $f=-\Delta u$ and $g=-\Delta
  v$. By \eqref{lap-uv-L1}, \eqref{combine}, and Corollary
  \ref{corstein}, $f$ and $g$ are $L^1(B_1(0))$ solutions of
  \eqref{hls_syst} for some positive constant $M$. Hence, by Theorem
  \ref{crucial}, $f$ and $g$ satisfy \eqref{e17} and \eqref{e18} with
  $s=n$. It therefore follows from \eqref{B-L} that
\begin{align*}
u(x)&={\cal O}\left(|x|^{-(n-2)}\right)
+o\left(|x|^{-\frac{\lambda(\alpha-2)(n-2)}{n}}\right)
\quad\text{as }x\to 0\\
v(x)&={\cal O}\left(|x|^{-(n-2)}\right)
+o\left(|x|^{-\frac{\lambda(\alpha-2)[\sigma(\beta-2)-2]}{n}}\right)
\quad\text{as }x\to 0
\end{align*}
which immediately gives Theorems \ref{thm3}--\ref{thm4B}.
\end{proof}

\section{Proof of Theorem \ref{optimalThm2}}

Define continuous functions $\varphi,\psi : (0,1)\to(0,1)$ by
\begin{equation}\label{B}
\varphi=\max\{h^{1/2},h^{1/(2\sigma)}\} \quad \text{and}\quad
\psi(t)=B\varphi(t)^\sigma t^{\frac{\lambda(\alpha-2)}{n}(n-\sigma(\beta-2))}
\end{equation}
where $B=B(n,\beta,\sigma)$ is a positive constant to be specified
later.

Let $\{x_j\}\subset\R^n$ be a sequence satisfying
\[
0<4|x_{j+1}|<|x_j|<1/2,\quad\sum_{j=1}^\infty\varphi(|x_j|)<\infty,\quad\text{and}
\quad \sum_{j=1}^\infty\psi(|x_j|)<\infty,
\]
and let $r_j=|x_j|^{\lambda(\alpha-2)/n}$. Then by Lemma \ref{optimal}
there exist a positive constant $A=A(n)$ and positive functions
$u,\hat v\in C^\infty(\R^n\setminus\{0\})$  such that
\eqref{ch3_gr6}--\eqref{ch3_gr9} hold as stated and also with $u$ and
$\varphi$ replaced with $\hat v$ and $\psi$ respectively. Let
\[
v=\hat v +|x|^{-(n-2)}.
\]
As in the proof of Theorem \ref{optimalthm1}, we can modify $u$ and
$v$ on $\R^n\setminus B_2(0)$ in such a way that they become
$C^\infty(\R^n\setminus\{0\})\cap L^1(\R^n)$ functions, and, by
\eqref{ch3_gr7}, $u$ and $v$ will satisfy \eqref{nonloc} in
$B_2(0)\setminus\{0\}$ provided they satisfy \eqref{nonloc} in
$\cup_{j=1}^\infty B_{r_j}(x_j)$.

Since, as the proof of Lemma \ref{int1} shows,
\[
\intl_{|y|<2}\frac{1}{|x-y|^\alpha}\frac{1}{|y|^{n-2}}\,dy \ge
\frac{C}{|x|^{\alpha-2}}\quad\text{for }0<|x|<2,
\]
we have for $|x-x_j|<r_j$ that
\begin{align*}
\left(\frac{1}{|x|^\alpha}*v\right)^\lambda
&\ge\left(\intl_{|y|<2}\frac{1}{|x-y|^\alpha}\frac{1}{|y|^{n-2}}\,dy\right)^\lambda\\
&\ge \frac{C}{|x_j|^{\lambda(\alpha-2)}}>\frac{\varphi(|x_j|)}{r_j^n}\ge-\Delta u.
\end{align*}
Moreover,
\begin{align*}
\frac{u(x_j)}{h(|x_j|)|x_j|^{-\lambda(\alpha-2)(n-2)/n}}
&\ge
\frac{A\varphi(|x_j|)r_j^{-(n-2)}}{\varphi(|x_j|)^2r_j^{-(n-2)}}\\
&=\frac{A}{\varphi(|x_j|)}\to\infty\quad\text{as }j\to\infty.
\end{align*}
Also, for $|x-x_j|<r_j$,
\begin{align*}
\left(\frac{1}{|x|^\beta}*u\right)^\sigma
&\ge\left(\intl_{|y-x_j|<r_j}\frac{u(y)}{|x-y|^\beta}\,dy\right)^\sigma\\
&\ge \left(\frac{A\varphi(|x_j|)}{r_j^{n-2}}
\frac{|B_{r_j}(x_j)|}{(2r_j)^\beta}\right)^\sigma
=B\varphi(|x_j|)^\sigma r_j^{-(\beta-2)\sigma}
\end{align*}
where $B=(2^{-\beta}|B_1(0)|A)^\sigma$. Hence \eqref{B} implies
\begin{align*}
\left(\frac{1}{|x|^\beta}*u\right)^\sigma
&\ge \psi(|x_j|)r_j^{-(n-\sigma(\beta-2))}r_j^{-(\beta-2)\sigma}\\
&=\psi(|x_j|)r_j^{-n}\ge -\Delta \hat v=-\Delta v\quad\text{for }|x-x_j|<r_j.
\end{align*}
Finally, again by \eqref{B},
\begin{align*}
\frac{v(x_j)}{h(|x_j|)|x_j|^{-\frac{\lambda(\alpha-2)[\sigma(\beta-2)-2]}{n}}}
&\ge \frac{A\psi(|x_j|)r_j^{-(n-2)}}{\varphi(|x_j|)^{2\sigma}r_j^{-(\sigma(\beta-2)-2)}}\\
&=\frac{AB}{\varphi(|x_j|)^\sigma}\to\infty\quad\text{as }j\to\infty
\end{align*}
and
\[
\frac{v(x_j)}{h(|x_j|)|x_j|^{-(n-2)}}\ge \frac{1}{h(|x_j|)}\to\infty
\quad\text{as }j\to\infty.
\]
This completes the proof of Theorem \ref{optimalThm2}.

\section{Proof of Theorem \ref{thm5}}
Choose $M>1$ such that $\int_{|y|<2}M|x-y|^{-\b}\,dy>1$ for
$|x|<2$. The positive functions $u$ and $v$ that we construct will
satisfy not just \eqref{nonloc}, \eqref{large1}, and \eqref{large2}
but also
\begin{equation}\label{M}
u\ge M \quad\text{in } B_2(0)\setminus\{0\}.
\end{equation}
If $u$ and $v$ are positive functions satisfying \eqref{nonloc} and
\eqref{M} then $u$ and $v$ also satisfy \eqref{nonloc} for any larger
value of $\sigma$ because then
\[
\int_{\R^n}\frac{u(y)\,dy}{|x-y|^\b}\ge\int_{|y|<2}\frac{M\,dy}{|x-y|^\b}
>1\quad\text{for }0<|x|<2.
\]
Hence we can assume for the proof of Theorem \ref{thm5} that
\neweq{aa} \sigma<\frac{n}{\beta-2}.
\endeq

Define
$$
a=\frac{1}{\lambda(\alpha-2)-n}\quad \mbox{ and }\quad b=\frac{1}{n-\sigma(\beta-2)}.
$$
Using \eq{optimm1} and \eq{aa} we have
\neweq{aa1}
a,b>0\quad\mbox{ and }\quad a\lambda-b<0.
\endeq
Let $\varphi:(0,1)\to (0,1)$, $\{x_j\}\subset \R^n$, $\{r_j\}\subset
(0,1)$ and $A=A(n)$ be as in Lemma \ref{optimal}. By Lemma
\ref{optimal} there exists a positive function $u\in
C^\infty(\R^n\setminus\{0\})$ that satisfies
\eq{ch3_gr6}--\eq{ch3_gr9}.

Since adding a positive constant to $u$ will not change the fact that
$u$ satisfies \eq{ch3_gr6}--\eq{ch3_gr9}, we can assume, instead of
\eqref{ch3_gr8}, that $u>M$ in $B_2(0)\setminus\{0\}$ where $M$ is as
stated above. As in the proof of Theorem \ref{optimalthm1},
we modify $u$ on $\R^n\setminus B_2(0)$ in
such a way as to obtain a $C^\infty(\R^n\setminus\{0\})\cap L^1(\R^n)$
function. For every $j\geq 1$ we define $\psi_j$ as a function of $r_j$
by \neweq{phij} r_j=\left[\frac{(B\psi_j)^\lambda}{\varphi(|x_j|)}
\right]^a\;,\quad\mbox{where}\quad B=B(n)=\frac{A|B_1(0)|}{2^n}>0.
\endeq
Note that
$$
\frac{A\psi_j}{r_j^{n-2}}=\frac{A}{B^{a\lambda(n-2)}}
\frac{\varphi(|x_j|)^{a(n-2)}}{\psi_j^{\lambda a(n-2)-1}}.
$$
By decreasing $r_j$ (and thereby decreasing $\psi_j$) we may assume
$$
\sum_{j=1}^\infty\psi_j<\infty,
$$
\neweq{c1}
\psi_j^{a\lambda-b}
\geq \frac{\varphi(|x_j|)^{a-b\sigma}}{B^{a\lambda+b\sigma}},
\endeq
and
\neweq{c2}
\frac{A\varphi(|x_j|)}{r_j^{n-2}}>>h(|x_j|)\;,\quad \frac{A\psi_j}{r_j^{n-2}}>>h(|x_j|)\quad\mbox{ as }j\to\infty.
\endeq
It follows from \eqref{c1} and \eqref{phij} that
\begin{equation}\label{help}
\left(\frac{B\varphi(|x_j|)}{r_j^{\b-2}}\right)^\sigma\ge
\frac{\psi_j}{r_j^n}.
\end{equation}
Let $\psi:(0,1)\to (0,1)$ be a continuous function such that $\psi(|x_j|)=\psi_j$. By Lemma \ref{optimal} there exists a positive function $v\in C^\infty(\R^n\setminus\{0\})$ such that
\neweq{ch3_gr10}
0\leq -\Delta v\leq \frac{\psi(|x_j|)}{r_j^n}\quad\mbox{ in }B_{r_j}(x_j),
\endeq
\neweq{ch3_gr11}
 -\Delta v=0\quad\mbox{ in }\R^n\setminus\Big(\{0\}\cup \bigcup_{j=1}^\infty B_{r_j}(x_j)\Big),
\endeq
\neweq{ch3_gr12}
v\geq 1\quad\mbox{ in }\R^n\setminus\{0\},
\endeq
\neweq{ch3_gr13}
v\geq \frac{A\psi(|x_j|)}{r_j^{n-2}} \quad\mbox{ in } B_{r_j}(x_j).
\endeq
We modify $v$ on $\R^n\setminus B_2(0)$ in such a way as to obtain a
$C^\infty(\R^n\setminus\{0\})\cap L^1(\R^n)$ function.
In order to check that
$u$ and $v$ satisfy
 \eq{nonloc} let us remark first that by \eq{ch3_gr13} we have
\neweq{check1}
\begin{aligned}
\int\limits_{\R^n} \frac{v(y)\,dy}{|x-y|^{\alpha}} &\geq \frac{A\psi(|x_j|)}{r_j^{n-2}}  \int\limits_{B_{r_j}(x_j)}   \frac{dy}{|x-y|^{\alpha}}\\
& \geq
\frac{A\psi(|x_j|)}{r_j^{n-2}}\frac{|B_{r_j}(x_j)|}{(2r_j)^\alpha}\geq
\frac{B\psi(|x_j|)}{r_j^{\alpha-2}}\quad\mbox{ for }x\in
B_{r_j}(x_j)
\end{aligned}
\endeq
 and similarly
\neweq{check2}
\int\limits_{\R^n} \frac{u(y)\,dy}{|x-y|^{\beta}}\geq
\frac{B\varphi(|x_j|)}{r_j^{\beta-2}}\quad\mbox{ for }x\in
B_{r_j}(x_j).
\endeq
Now, by \eq{ch3_gr6}, \eq{phij}, \eqref{help}, \eq{ch3_gr10},
\eq{check1} and \eq{check2} we deduce that $u$ and $v$ are solutions
of \eq{nonloc}. Finally, to check that $u$ and $v$ satisfy \eq{large1}
and \eq{large2} along the sequence $\{x_j\}$ we use \eq{ch3_gr9},
\eq{ch3_gr13} and \eq{c2}.

\section{Proof of Theorem \ref{optimal-fg}}
We consider two cases.
\medskip

\noindent{\bf Case I}.
Suppose $\lambda(\alpha-2)<n$. Let $\chi:\R^n\to [0,1]$ be a
$C^\infty$ function such that $\chi=1$ for $|x|<2\varepsilon$ and
$\chi=0$ for $|x|>4\varepsilon$. Then
$f(x):=\varepsilon|x|^{-\lambda(\alpha-2)}\chi(x)$  and
$g(x):=\varepsilon|x|^{-\sigma(\beta-2)}\chi(x)$
clearly satisfy \eqref{op1}--\eqref{op4}.
\medskip

\noindent{\bf Case II}.
Suppose $\lambda(\alpha-2)\ge n$. Define $\varphi:(0,1)\to (0,1)$ by
$
\varphi=h^{n/(2\sigma(n+2-\beta))}.
$
Let $\{x_j\}$ be a sequence in $\R^n$ such that
\[
0<4|x_{j+1}|<|x_j|<\varepsilon/2.
\]
Let
\begin{equation}\label{op5}
\ep_j=\vphi(|x_j|)\quad\text{and }\quad
r_j=\left(\frac{\ep_j}{\ep}\right)^{1/n}(2|x_j|)^{\la(\a-2)/n}.
\end{equation}
By taking a subsequence we can assume
\begin{equation}\label{op6}
0<r_j<|x_j|/2<1\quad\text{and}\quad \ep_j<2^{-j}.
\end{equation}
Thus
\begin{equation}\label{op7}
\sum_{j=1}^\infty\ep_j<\infty.
\end{equation}
Let
\begin{equation}\label{op8}
\delta_j=\ep J^\sigma\ep_j^\sigma r_j^{n-\sigma(\b-2)}
\end{equation}
where $J=J(n)>0$ is a constant to be specified later. By \eqref{op6},
$\delta_j\le \ep J^\sigma2^{-\sigma j}$ and hence
\begin{equation}\label{op9}
\sum_{j=1}^\infty\delta_j<\infty.
\end{equation}
Define sequences $\{M_j\}$ and $\{N_j\}$ by
\begin{equation}\label{op10}
M_j=\frac{\ep_j}{r_j^n}\quad\text{and}\quad N_j=\frac{\delta_j}{r_j^n}.
\end{equation}
Then by \eqref{op5} and \eqref{op6},
\begin{equation}\label{op11}
M_j=\frac{\ep}{(2|x_j|)^{\la(\a-2)}}<\frac{\ep}{|x|^{\la(\a-2)}}
\quad\text{for } |x-x_j|<r_j.
\end{equation}
Let $\psi:\R^n\to[0,1]$ be a $C^\infty$ function such that $\psi=0$ in
$\R^n\setminus B_1(0)$ and $\psi(0)=1$. Define $\psi_j:\R^n\to[0,1]$
by $\psi_j(y)=\psi(\eta)$ where $y=x_j+r_j\eta$. Then
 \begin{equation}\label{op12}
 \psi_j(x_j)=1\quad\text{and }\quad
 \intl_{\mathbb{R}^n} \psi_j (y)\,dy=\intl_{\mathbb{R}^n} \psi(\eta)r^{n}_{j} \,d\eta=r^{n}_{j}I
 \end{equation}
 where $I=\intl_{\mathbb{R}^n}\psi(\eta)\,d\eta>0$.

Define $f,g:\R^n\setminus\{0\}\to[0,\infty)$ by
\[
f=\sum_{j=1}^\infty M_j\psi_j \quad\text{and} \quad g=\sum_{j=1}^\infty N_j\psi_j.
\]
Since the functions $\psi_j$ have disjoint supports, $f,g\in C^\infty
(\R^n \backslash\{0\})$ and by \eqref{op12} and \eqref{op10}
we have
\[
\intl_{\R^n}f(y)\,dy=I\sum_{j=1}^\infty\ep_j\quad\text{and}\quad
\intl_{\R^n}g(y)\,dy=I\sum_{j=1}^\infty\delta_j.
\]
Thus, by \eqref{op7} and \eqref{op9}, we see that $f,g\in L^1(\R^n)$.

From \eqref{op12} and \eqref{op11} we have
\[
|x_j|^{\la(\a-2)}f(x_j)=M_j|x_j|^{\la(\a-2)}=\frac{\ep}{2^{\la(\a-2)}}
\]
and
\[
|x|^{\la(\a-2)}f(x)\le M_j|x|^{\la(\a-2)}<\ep \quad\text{for }|x-x_j|<r_j.
\]
Thus $f$ satisfies \eqref{op3} and the first line of
\eqref{op2}. (Note that we only need to check \eqref{op2} holds in
$\cup_{j=1}^\infty B_{r_j}(x_j)$ because elsewhere $f=g=0$.)

For $x=x_j +r_j \xi$ and $|\xi|<1$ we have
 \begin{align*}
   \Bigg(\intl_{|y|<\ep} &\frac{f(y)\,dy}{|x-y|^{\b-2}}\Bigg)^\sigma
 \geq  \left(    \intl_{|y-x_j
    |<r_j}
\frac{M_j \psi_j (y)\,dy}{|x-y|^{\b-2}}   \right)   ^\sigma
  =\left(   \intl_{|\eta|<1}
\frac{M_j \psi(\eta)r^{n}_{j}\,d\eta}{r^{\b-2}_{j} |\xi-\eta|^{\b-2}}
\right)   ^\sigma\\
  &=\left(  \frac{\varepsilon_j}{r^{\b-2}_j} \intl_{|\eta|<1}
\frac{\psi(\eta)\,d\eta}{|\xi-\eta|^{\b-2}}   \right)  ^\sigma\\
  &\geq \left(\frac{J\varepsilon_j}{r^{\b-2}_{j}}\right)^\sigma \quad\text{ where }J=\min_{|\xi|\leq 1} \intl_{|\eta|<1} \frac{\psi(\eta)\,d\eta}{|\xi-\eta|^{\b-2}}>0\\
&=\frac{1}{\ep}\frac{\delta_j}{r_j^n}=\frac{1}{\ep}N_j\ge\frac{1}{\ep}g(x)
\end{align*}
by \eqref{op8}. Thus the second line of \eqref{op2} holds.

Finally, by \eqref{op8} and \eqref{op5},
\begin{align*}
g(x_j)&=N_j=\frac{\delta_j}{r_j^n}=\frac{C\ep_j^\sigma}{r_j^{\sigma(\b-2)}}\\
&=\frac{C\ep_j^\sigma}{\ep_j^{\sigma(\b-2)/n}|x_j|^{\la(\a-2)\sigma(\b-2)/n}}\\
&=\frac{C\ep_j^{\sigma(n+2-\b)/n}}{|x_j|^{\la(\a-2)\sigma(\b-2)/n}}
=\frac{C\sqrt{h(|x_j|)}}{|x_j|^{\la(\a-2)\sigma(\b-2)/n}}
\end{align*}
which gives \eqref{op4}.

\section{Proof of Theorem \ref{large-fg}}\label{sec-large-fg}

The functions $f$ and $g$ that we construct will satisfy not just
\eqref{A2} and \eqref{op2} but also
\begin{equation}\label{A5}
f=g=0\quad\text{in }\R^n\setminus B_\ep(0).
\end{equation}
If $f$ and $g$ satisfy \eqref{A2}, \eqref{op2}, and \eqref{A5} then
they also satisfy \eqref{op2} for any larger value of $\sigma$. Hence
we can assume without loss of generality that
\begin{equation}\label{A6}
\sigma<\frac{n}{\b-2}.
\end{equation}

Let $\{x_j\}\subset\R^n$ and $\{\ep_j\}\subset (0,1)$ be sequences in such that
\[
0<4|x_{j+1}|<|x_j|<\varepsilon/2
\]
and
\begin{equation}\label{A7}
\sum_{j=1}^\infty\ep_j<\infty.
\end{equation}
Choose $r_j\in(0,|x_j|/2)$ such that
\[
N_j:=\ep\left(\frac{J\ep_j}{r_j^{\b-2}}\right)^ \sigma\ge h(|x_j|)^2, \qquad
M_j:=\frac{\ep_j}{r_j^n}\ge h(|x_j|)^2,
\]
\[
\delta_j:=\ep J^\sigma\ep_j^\sigma r_j^{n-\sigma(\b-2)}<2^{-j},
\]
and
\[
r_j^{\la(\b-2)[\sigma-(\frac{n+2-\a}{\b-2}+\frac{n}{\b-2}\frac{1}{\la})]}
\le \ep^{\la+1}J^{\la(\sigma+1)}\ep_j^{\la\sigma-1}
\]
where $J=J(n)>0$ is a constant to specified later. (This is possible
because the exponents on $r_j$ in all of these conditions are positive
by \eqref{A1} and \eqref{A6}.) Then
\begin{equation}\label{A8}
h(|x_j|)^2\le
N_j=\frac{\delta_j}{r_j^n}=\ep\left(\frac{J\ep_j}{r_j^{\b-2}}\right)^\sigma,
\end{equation}
\begin{equation}\label{A9}
h(|x_j|)^2\le
M_j=\frac{\ep_j}{r_j^n}\le \ep\left(\frac{J\delta_j}{r_j^{\a-2}}\right)^\la,
\end{equation}
and
\begin{equation}\label{A10}
\sum_{j=1}^\infty\delta_j<\infty.
\end{equation}
Define $\psi$, $\psi_j$, $I$, $J$, $f$, and $g$ as in the proof of
Theorem \ref{optimal-fg}. Then $f$ and $g$ satisfy
\eqref{A5}. Moreover, using \eqref{A7}--\eqref{A10}, we see as in the
proof of Theorem \ref{optimal-fg} that \eqref{A2} holds,
\[
\left(\intl_{|y|<\ep} \frac{f(y)\,dy}{|x-y|^{\b-2}}\right)^\sigma
 \ge \frac{1}{\ep}g(x)\quad\text{for } x\in\R^n\setminus\{0\}
\]

and
\[
\left(\intl_{|y|<\ep} \frac{g(y)\,dy}{|x-y|^{\a-2}}\right)^\la
 \ge \frac{1}{\ep}f(x)\quad\text{for } x\in\R^n\setminus\{0\}.
\]
Thus $f$ and $g$ satisfy \eqref{op2}. Also
\[
f(x_j)=M_j\ge h(|x_j|)^2 >>h(|x_j|)\quad\text{as }j\to\infty
\]
and
\[
g(x_j)=N_j\ge h(|x_j|)^2>>h(|x_j|)\quad\text{as }j\to\infty.
\]
Hence \eqref{A3} and \eqref{A4} hold.

\medskip

\noindent{\bf Acknowledgement.} The authors would like to thank
Stephen J. Gardiner for helpful discussions.


\begin{thebibliography}{10}



\bibitem{BR1996} M.-F. Bidaut-V\'eron and Th. Raoux, Asymptotics of solutions
  of some nonlinear elliptic systems, Comm. Partial Differential
  Equations 21 (1996), 1035--1086.

\bibitem{BL1981} H. Brezis and P-L. Lions, A note on isolated singularities for linear elliptic equations, Mathematical analysis and applications, Part A, pp. 263-266, Adv. in Math. Suppl. Stud., 7a, Academic Press, New York-London, 1981.

\bibitem{CDM2008} G. Caristi, L. D'Ambrosio and E. Mitidieri,
Representation formulae for solutions to some classes of higher
order systems and related Liouville theorems, Milan J. Math. 76
(2008), 27-67.

\bibitem{CL2005} W. Chen, C. Li and  B. Ou, Classification of solutions for a system of integral equations,
Comm. Partial Differential Equations 30 (2005),  59-65.

\bibitem{CL2009} W. Chen and C. Li, An integral system and the Lane-Emden conjecture,
Discrete Contin. Dyn. Syst. 24 (2009),  1167-1184.

\bibitem{DA2010} J. T. Devreese and A. S. Alexandrov, Advances in polaron physics, Springer Series in Solid-State Sciences, vol. 159, Springer, 2010.

\bibitem{GMT2011} M. Ghergu, A. Moradifam and S.D. Taliaferro, Isolated singularities of polyharmonic inequalities,
J. Functional Anal. 261 (2011), 660-680.


\bibitem{GTV2014} M. Ghergu, S.D. Taliaferro and  I.E. Verbitsky, Pointwise bounds and blow-up for systems of semilinear elliptic inequalities at an isolated singularity via nonlinear potential estimates, http://arxiv.org/abs/1402.0113

\bibitem{GT1983} D. Gilbarg and N.S. Trudinger, Elliptic Partial Differential Equations of Second Order, 2nd edition, Springer-Verlag, Berlin, 1983.


\bibitem{HM1972} V.P. Havin and V.G. Maz'ya, Nonlinear potential theory, Usp. Mat. Nauk 27 (1972), 67-138 (in Russian). English translation: Russ. Math. Surv. 27 (1972), 71-148.

\bibitem{H1972} L.I. Hedberg, On certain convolution inequalities, Proc. Amer. Math. Soc. 36 (1972), 505-510.

\bibitem{JL2006} C. Jin and C. Li,
Quantitative analysis of some system of integral equations,
Calc. Var. Partial Differential Equations 26 (2006), 447-457.


\bibitem{J1995} K.R.W. Jones, Newtonian quantum gravity, Australian Journal of Physics 48 (1995),  1055-1081.

\bibitem{Leia2013} Y. Lei, On the regularity of positive solutions of a class of Choquard type equations, Math. Z. 273 (2013), 883-905.

\bibitem{Leib2013} Y.  Lei, Qualitative analysis for the static Hartree-type equations, SIAM J. Math. Anal. 45 (2013), 388-406.


\bibitem{LLM2012} Y. Lei, C. Li and  C. Ma, Asymptotic radial symmetry and growth estimates of positive solutions to weighted Hardy-Littlewood-Sobolev system of integral equations,
Calc. Var. Partial Differential Equations 45 (2012),  43-61.

\bibitem{L1976} E.H. Lieb, Existence and uniqueness of the minimizing solution of Choquard's nonlinear equation, Studies in Appl. Math. 57 (1976/77), 93-105.

\bibitem{MPT1998}  I.M. Moroz, R. Penrose and  P. Tod, Spherically-symmetric solutions of the Schro\"odinger-Newton equations, Classical Quantum Gravity 15 (1998), 2733-2742.


\bibitem{P1954} S. Pekar, Untersuchung \"uber die Elektronentheorie der Kristalle, Akademie Verlag, Berlin, 1954.

\bibitem{stein} E.M. Stein, Singular Integrals and Differentiability Properties of Functions, Princeton University Press, 1970.

\bibitem{S2009} Ph. Souplet, The proof of the Lane-Emden conjecture in
  four space dimensions,  Adv. Math. 221 (2009), 1409--1427.

\bibitem{T2007} S.D. Taliaferro, Isolated singularities of nonlinear
  parabolic inequalities, Math. Ann. 338 (2007), 555-586.

\bibitem{T2011} S.D. Taliaferro, Initial blow-up of solutions of semilinear parabolic inequalities, J. Differential Equations 250 (2011), 892-928.

\end{thebibliography}
\end{document}